\documentclass[10pt,a4paper]{amsart}
\usepackage[utf8]{inputenc}
\usepackage[T1]{fontenc}
\usepackage[french, english]{babel}
\usepackage{amsmath,amsfonts,amssymb,amsthm}
\usepackage{stmaryrd} 
\usepackage{array}
\usepackage{enumerate}
\usepackage{hyperref}
\usepackage{mathabx}
\usepackage{amsbsy}
\usepackage{lineno}
\newcommand{\Q}{\mathbb{Q}}
\newcommand{\N}{\mathbb{N}}
\newcommand{\Z}{\mathbb{Z}}
\newcommand{\R}{\mathbb{R}}

\newcommand{\C}{\mathbb{C}}
\newcommand{\Gal}{\mathrm{Gal}}

\newcommand{\G}{\mathbb{G}}
\newcommand{\poubelle}[1]{}

\theoremstyle{plain}
\newtheorem{thm}{Theorem}[section]
\newtheorem{prop}[thm]{Proposition}
\newtheorem{cor}[thm]{Corollary}
\newtheorem{lmm}[thm]{Lemma}

\newtheorem{conj}[thm]{Conjecture}

\newtheorem*{lmm*}{Lemma}
\newtheorem*{conj*}{Conjecture}
\newtheorem*{thm*}{Theorem}
\newtheorem*{claim*}{Fait}

\theoremstyle{plain}
\newtheorem{rqu}[thm]{Remark}

\theoremstyle{plain}
\newtheorem{defn}[thm]{Definition}

\usepackage[all]{xy}
\usepackage{tikz}
\usetikzlibrary{shadows,decorations.pathmorphing}
\usepackage{comment}
\usepackage{ifthen}

\begin{document}

\poubelle{\linenumbers}

\title[New way to tackle a conjecture of R\'emond]{A new way to tackle a conjecture of R\'emond} 
\date{}
\author{Arnaud Plessis}
\address{Arnaud Plessis : Morningside Center of Mathematics, Academy of Mathematics and Systems Science, Chinese Academy of Sciences, Beijing 100190}
\email{plessis@amss.ac.cn} 
\begin{abstract}
Let $\Gamma\subset \overline{\Q}^*$ be a finitely generated subgroup.
Denote by $\Gamma_\mathrm{div}$ its division group.  
A recent conjecture due to R\'emond, related to the Zilber-Pink conjecture, predicts that the absolute logarithmic Weil height of an element of $\Q(\Gamma_\mathrm{div})^*\backslash \Gamma_{\mathrm{div}}$ is bounded from below by a positive constant depending only on $\Gamma$.
In this paper, we propose a new way to tackle this problem. 
\end{abstract}
\maketitle 

\section{Introduction} \label{Introduction}
\subsection{R\'emond's conjecture} 
Let $\alpha\in\overline{\Q}$. 
The (absolute logarithmic) Weil height of $\alpha$ is the real number \[h(\alpha) = \frac{1}{[\Q(\alpha) : \Q]} \log\left(\vert lc(\alpha)\vert \prod_\sigma \max\{1, \vert \sigma\alpha\vert\}\right), \] where $lc(\alpha)$ denotes the leading coefficient of the minimal polynomial of $\alpha$ over $\Z$, and where $\sigma$ runs over all field embeddings $\Q(\alpha) \to \overline{\Q}$. 
The function $h : \overline{\Q} \to \R$ is non-negative and vanishes precisely at $\mu_\infty$, the set of all roots of unity, and $0$ by a theorem of Kronecker. 
It is also invariant under Galois conjugation and satisfies $h(\alpha^n)=\vert n\vert h(\alpha)$ for all $\alpha\in\overline{\Q}$ and all $n\in\Z$. 
For more properties on $h$, see \cite{BombieriGubler}.
 
Let $X$ be a set of algebraic numbers.
We say that points of small height (or short, small points) of $X$ lie in a set $Y\subset \overline{\Q}$ if there exists a positive constant $c$ such that $h(\alpha) \geq c$ for all $\alpha\in X\backslash Y$.  

The case where small points of an algebraic field lie in $\mu_\infty\cup \{0\}$ has been intensively studied since the beginning of this century, see for example \cite{BombieriZannier, AmorosoDvornicich, AmorosoZannier10, Habegger, AmorosoDavidZannier, Grizzard, FiliMilner, Galateau, Sahu, Plessismino, PazukiPengo, PlessisgenHab, Frey}. 
For various technical reasons, it is difficult to adapt the ideas of these papers to locate small points of an algebraic field that are not contained in $\mu_\infty \cup \{0\}$. 
Up to my knowledge, the first one who managed to do this is Amoroso in 2016 \cite{Amoroso}, see our explanation below for more details. 

Consider the field $L=\Q(\mu_\infty, \alpha, \alpha^{1/2},\alpha^{1/3},\dots)$, where $\alpha\in \overline{\Q}^*$.
It is trivial to construct small points in $L^*$. 
For instance, we have all roots of unity, but also $\alpha^{1/n}$ with $n$ large enough since $h(\alpha^{1/n})=h(\alpha)/n$. 
But does $L^*$ also contain non-obvious small points; i.e. elements of small height not in $\{\zeta \alpha^q, \zeta\in\mu_\infty, q\in\Q\}$?
By a result of Amoroso and Dvornicich \cite{AmorosoDvornicich}, the answer is no if $\alpha\in\mu_\infty$.  
But this question is still open for $\alpha\notin\mu_\infty$.
A very particular case of a deep and recent conjecture of R\'emond, related to the Zilber-Pink conjecture, predicts a negative answer to this question, see Conjecture \ref{conj 0} below.  

Throughout this introduction $\Gamma$ always denotes a finitely generated subgroup of $\overline{\Q}^*$.
Define $\Gamma_\mathrm{div}$ as the division group of $\Gamma$, i.e. the set of $\gamma\in\overline{\Q}^*$ for which there exists an integer $n\in\N=\{1,2,\dots\}$ satisfying $\gamma^n\in \Gamma$.

\begin{conj}[R\'emond, \cite{Remond}, Conjecture 3.4] \label{conj 0}
Let $\Gamma\subset \overline{\Q}^*$ be a finitely generated subgroup.
\begin{enumerate} [(i)]
\item (strong form): There is a positive constant $c$ such that 
\begin{equation*}
h(\alpha) \geq \frac{c}{[\Q(\Gamma_\mathrm{div},\alpha) : \Q(\Gamma_\mathrm{div})]} \; \text{for all} \; \alpha\in \overline{\Q}^*\backslash \Gamma_\mathrm{div}.
\end{equation*}
\item (weak form): For any $\varepsilon >0$, there exists a positive constant $c_\varepsilon$ such that
\begin{equation*}
h(\alpha)\geq \frac{c_\varepsilon}{[\Q(\Gamma_\mathrm{div},\alpha) : \Q(\Gamma_\mathrm{div})]^{1+\varepsilon}} \; \text{for all} \; \alpha\in \overline{\Q}^*\backslash \Gamma_\mathrm{div}.
\end{equation*}
\item (degree one form): Small points of $\Q(\Gamma_\mathrm{div})^*$ lie in $\Gamma_\mathrm{div}$.
\end{enumerate}
\end{conj}
The reader interested on recent advances concerning R\'emond's conjecture is referred to \cite{Remond, Amoroso, Grizzard2, Pottmeyer, Plessisloca, Plessisiso}.

Write $\langle X\rangle$ for the group generated by a set $X\subset\overline{\Q}^*$. 
We clearly have the implications $(i) \Rightarrow (ii) \Rightarrow (iii)$.
The strong form generalizes the relative Lehmer's problem, which corresponds to the case $\Gamma= \{1\}$. 
Up to my knowledge, the strong form is not yet known in any situation and both the weak form and the degree one form are only known when $\Gamma$ is trivial, see \cite{AmorosoZannier00} and \cite{AmorosoDvornicich}.
The first partial result going in the direction of Conjecture \ref{conj 0} for nontrivial groups was given by Amoroso.
He proved that small points of $\Q(\zeta_3, 2^{1/3}, \zeta_{3^2}, 2^{1/3^2},\dots)^*$, with $\zeta_n=e^{2i\pi/n}$, lie in $\langle 2\rangle_\mathrm{div}$ \cite[Theorem 1.3]{Amoroso}. 
The author then proved the same assertion by remplacing $2$ with $\alpha\in\Q^*$ and $3$ with a rational prime $p>2$ \cite[Th\'eor\`eme 1.8]{Plessisiso}. 

The proof of these last two results relies on the "classical method", i.e. the one to treat the case where small points of an algebraic field lie in $\mu_\infty \cup \{0\}$.
As already implied in \cite[Remark 3.4]{Amoroso}, it seems that this method is (very) limited to handle Conjecture \ref{conj 0} $(iii)$, which explains why it is still largely open.
So we need to tackle this conjecture from a totally different angle.
We suggest a new one below.

\subsection{Presentation of results} \label{subsection 1.2}
Let $(y_n)$ be a sequence of $\Q(\Gamma_{\mathrm{div}})^*$ such that $h(y_n)\to 0$. 
From Bilu's equidistribution theorem \cite{Bilu}, it is not hard to check that \[ \frac{\# \{\sigma : \Q(y_n)\hookrightarrow \overline{\Q}, \;  1-\varepsilon\leq \vert \sigma y_n \vert\leq 1+\varepsilon\}}{[\Q(y_n) : \Q]} \underset{n\to+\infty}{\longrightarrow} 1 \] for all $\varepsilon >0$ (if the sequence of terms $[\Q(y_n) : \Q]$ is bounded, then $y_n\in\mu_\infty$ for all $n$ large enough by Northcott's theorem, and so the ratio above is $1$ for all $n$ large enough). 
In other words, if the Weil height of $y_n$ is small enough, then "most of" its Galois conjugates over $\Q$ are "close" to the unit circle. 
Nonetheless, as $K_\Gamma=\Q(\mu_\infty, \Gamma)$ is not a number field, Bilu's theorem above cannot tell anything about the location in the complex plane of Galois conjugates of $y_n$ over $K_\Gamma$. 
Our results stipulate that a fairly precise knowledge of the distribution of these numbers allows us to solve a part of Conjecture \ref{conj 0} $(iii)$. 
In Section \ref{section 1.5}, we will prove a result reducing the study of Conjecture \ref{conj 0} $(iii)$ to the case that for all $n$ large enough, "most of" Galois conjugates of $y_n$ over $K_\Gamma$ are "close" to the unit circle. 
 
Let $\alpha\in\overline{\Q}$ and let $\varepsilon>0$. 
We write $O_\Gamma(\alpha)$ for the orbit of $\alpha$ under $\Gal(\overline{\Q}/K_\Gamma)$ and we put \[ d_{\Gamma, \varepsilon}(\alpha):= \frac{\# \{x\in O_\Gamma(\alpha), 1-\varepsilon\leq \vert x \vert^2\leq 1+\varepsilon\}}{\# O_\Gamma(\alpha)}. \]

\begin{thm} \label{thm0}
Let $\Gamma\subset \overline{\Q}^*$ be a finitely generated subgroup.
Then Conjecture \ref{conj 0} $(iii)$ is equivalent to the following assertion: Let $(y_n)$ be a sequence of $\Q(\Gamma_{\mathrm{div}})^*$ such that $h(y_n)\to 0$. 
Assume that $d_{\Gamma, \varepsilon}(y_n)\underset{n\to+\infty}{\longrightarrow} 1$ for all $\varepsilon>0$.
Then $y_n\in\Gamma_{\mathrm{div}}$ for all $n$ large enough.  
\end{thm}

The length of an element $x\in \Q(\Gamma_\mathrm{div})$ is defined to be the smallest integer $l\in\N$ for which $x$ can express as $x= \sum_{j=1}^l x_j \gamma_j$ with $x_j\in K_\Gamma$ and $\gamma_j\in\Gamma_\mathrm{div}$. 
Each element of $\Gamma_\mathrm{div}$ has length $1$. 
For $N\in\N$, put $l_N(\Gamma)$ the set of elements of $\Q(\Gamma_\mathrm{div})$ with length $\leq N$. 

From Theorem \ref{thm0}, it is enough to prove the following two conjectures to deduce Conjecture \ref{conj 0} $(iii)$. 

\begin{conj} \label{conj 0.5}
Let $\Gamma\subset\overline{\Q}^*$ be a finitely generated subgroup, and let $(y_n)$ be a sequence of $\Q(\Gamma_{\mathrm{div}})^*$ such that $h(y_n)\to 0$ and $d_{\Gamma,\varepsilon}(y_n)\underset{n\to+\infty}{\longrightarrow} 1$ for all $\varepsilon>0$.
Then there exists $N\in\N$ such that $y_n\in l_N(\Gamma)$ for all $n$ large enough.  
\end{conj}

\begin{conj} \label{conj 1}
Let $N\in\N$, and let $\Gamma\subset\overline{\Q}^*$ be a finitely generated subgroup. 
Let $(y_n)$ be a sequence of $l_N(\Gamma)\backslash \{0\}$ such that $h(y_n)\to 0$ and $d_{\Gamma,\varepsilon}(y_n)\underset{n\to+\infty}{\longrightarrow} 1$ for all $\varepsilon>0$. 
 Then $y_n\in\Gamma_{\mathrm{div}}$ for all $n$ large enough.  
\end{conj}

Conjecture \ref{conj 1} with $N=1$ was solved by R\'emond, see Lemma \ref{lmm 20} below.
Conjecture \ref{conj 0} $(iii)$ and Conjecture \ref{conj 0.5} with $N=1$ are therefore equivalent. 
However, this does not remove the interest of Conjecture \ref{conj 1} because Conjecture \ref{conj 0.5} could be easier to show when $N$ is large.

Both conjectures seem to be treated separately. 
Hence we focus in this article on the second one only.

For a set $X$ of algebraic numbers, write $X.\Gamma_{\mathrm{div}}$ for the set of $x\gamma$ with $x\in X$ and $\gamma\in \Gamma_{\mathrm{div}}$. 
Note that $l_1(\Gamma)= K_\Gamma.\Gamma_{\mathrm{div}}$. 
The rank of an abelian group $G$ is given by the maximal number of linearly independent elements in $G$. 
Note that $\Gamma$ and $\Gamma_{\mathrm{div}}$ have the same rank. 

\begin{lmm} [R\'emond] \label{lmm 20}
Let $\Gamma\subset\overline{\Q}^*$ be a finitely generated subgroup, and let $F/K_\Gamma$ be a finite extension.
Then small points of $F^*.\Gamma_{\mathrm{div}}$ lie in $\Gamma_{\mathrm{div}}$. 
\end{lmm}

\begin{proof}
By Amoroso and Zannier's result \cite{AmorosoZannier10}, small points of $F^*$ lie in $\mu_\infty\subset \Gamma_{\mathrm{div}}$. 
 By assumption, $\Gamma \subset \Gamma_{\mathrm{div}} \cap F^*\subset \Gamma_{\mathrm{div}}$. 
Thus $\Gamma, \Gamma_{\mathrm{div}}\cap F^*$ and $\Gamma_{\mathrm{div}}$ all have the same rank, which is finite since $\Gamma$ is finitely generated.
For $a\in F$, set \[h_{\Gamma}(a) = \min\{h(a\gamma), \gamma\in \Gamma_{\mathrm{div}}\}.\]
Thanks to a result of R\'emond \cite[Corollary 2.3]{Pottmeyer}, there is a positive constant $c$ such that $h_{\Gamma}(a)\geq c$ for all $a\in F^*\backslash \Gamma_{\mathrm{div}}$. 

Let $y= a\gamma\in F^*.\Gamma_{\mathrm{div}}$, where $a\in F^*$ and $\gamma\in \Gamma_{\mathrm{div}}$, such that $h(y)<c$. 
As $c > h(y)\geq h_{\Gamma}(a)$, we get $a\in \Gamma_{\mathrm{div}}$ by the foregoing. 
The lemma follows.
\end{proof}

Among all our theorems, the next one is the most difficult (and technical) to show. 
Hence we will prove it at the end of this paper, that is in Section \ref{section 5}. 

\begin{thm} \label{thm 2}
Let $\Gamma\subset \overline{\Q}^*$ be a finitely generated subgroup, let $N\in\N$, and let $(y_n)$ be a sequence of $l_N(\Gamma)$.
Assume that $d_{\Gamma, \varepsilon}(y_n)\underset{n\to+\infty}{\longrightarrow} 1$ for all $\varepsilon>0$. 
Then for all $\varepsilon>0$, we have $d_{\Gamma, \varepsilon}(y_n)=1$ for all $n$ large enough.
\end{thm}

\begin{rqu}
\rm{The proof of this theorem does not hold if we study the Galois conjugates of $y_n$ over $\Q$ (or $\Q(\Gamma)$) instead of $K_\Gamma$. 
The reason is that we will use Kummer theory and for this, it is primordial that our ground field contains all roots of unity.}
\end{rqu}

Theorem \ref{thm 2} means that if "most of" elements in $O_\Gamma(y_n)$ are "close" to the unit circle, then they all are. 
Note that this theorem holds regardless of the value of $h(y_n)$, which leads to the following natural question: How are the elements of $O_\Gamma(y_n)$ scattered around the unit circle when $h(y_n)$ is small enough ?  
We predict they are concyclic and located on a circle centered at the origin. 
If it checks out, then Conjecture \ref{conj 1} would immediately fall thanks to our next theorem, a proof of which is given in Section \ref{section X}.

Denote by $U(\Gamma)$ the set of algebraic numbers $\alpha$ such that the elements of $O_\Gamma(\alpha)$ are concyclic and located on a circle centered at the origin.  
Check that $\Gamma_\mathrm{div}$ is a subgroup of $U(\Gamma)$: Let $x\in\Gamma_{\mathrm{div}}$, and let $\sigma\in \Gal(\overline{\Q}/K_\Gamma)$. 
By definition of the division group, we have $x^n\in \Gamma$ for some $n\in\Z\backslash \{0\}$. 
Thus $\sigma(x^n)=x^n$ implying $\vert \sigma x\vert = \vert x\vert$; whence $x\in U(\Gamma)$.

\begin{thm} \label{thm 1}
Let $\Gamma\subset \overline{\Q}^*$ be a finitely generated subgroup, and let $N\in\N$.
Then small points of $l_N(\Gamma)\cap U(\Gamma)$ lie in $\Gamma_{\mathrm{div}}$. 
\end{thm}

End this section by providing a collection of $\Gamma$ for which small points of $l_N(\Gamma)\backslash \{0\}$ lie in $U(\Gamma)$, thus giving credit to our approach to attack Conjecture \ref{conj 0} $(iii)$. 

A CM-field is a totally imaginary quadratic extension of a totally real field. 
The maximal totally real field extension $\Q^{tr}$ of $\Q$ having only one quadratic extension, namely $\Q^{tr}(i)$, the CM-fields are therefore the subfields of $\Q^{tr}(i)$ that are not totally real. 
In particular, a compositum of CM-fields is a CM-field.
The classification of CM-number fields was made in \cite{Blanksby}; these are the fields of the form $\Q(\alpha)$ with $\alpha\in U\backslash \{\pm 1\}$, where $U$ denotes the set of algebraic numbers with all its conjugates over $\Q$ on the unit circle.
Note that if $\Gamma\subset U$, then $\Gamma_{\mathrm{div}}\subset U$. 
The field $\Q(\Gamma_{\mathrm{div}})$ is therefore a CM-field since it is the compositum of all $\Q(\alpha)$ with $\alpha\in\Gamma_{\mathrm{div}}\backslash\{\pm 1\}$. 

\begin{cor} \label{cor 2}
Let $\Gamma\subset U$ be a finitely generated subgroup, and let $N\in\N$. 
Then small points of $l_N(\Gamma)\backslash \{0\}$ lie in $\Gamma_{\mathrm{div}}$. 
\end{cor}

\begin{proof}
By a theorem of Schinzel \cite[Theorem 2]{Schinzel}, see also \cite{Pottmeyer2}, small points of $\Q^{tr}(i)^*$ lie in $U$.
As $\Gamma\subset U$, the field $\Q(\Gamma_{\mathrm{div}})$ is a CM-field; whence $\Q(\Gamma_{\mathrm{div}})\subset \Q^{tr}(i)$. 
Small points of $l_N(\Gamma)\backslash \{0\}$ therefore lie in $U$.
Corollary \ref{cor 2} now arises from Theorem \ref{thm 1} since $U \subset U(\Gamma)$ by definition. 
\end{proof}

Conjecture \ref{conj 0} $(iii)$ and Conjecture \ref{conj 0.5} are therefore equivalent for CM-fields.
This corollary is the second partial result going in the direction of Conjecture \ref{conj 0} $(iii)$, and the first one of this form. 

\subsection*{Acknowledgement} I thank Y. Bilu and M. Sombra for their helpful advice; F. Amoroso, L. Pottmeyer, G. R\'emond and the anonymous referees for their relevant suggestions and S. Fischler for his help concerning Lemma \ref{lmm Fischler}.
This work was funded by Morningside Center of Mathematics, CAS. 

\section{Kummer Theory} \label{section 1.25}
Fix once and for all a finitely generated subgroup $\Gamma\subset\overline{\Q}^*$ of rank $b$. 
As Conjecture \ref{conj 0} $(iii)$, and therefore Conjecture \ref{conj 1}, is true for $b=0$ (see the introduction), we can reduce the proof of all theorems mentioned in the introduction to the case that $b>0$.
We also fix once and for all a generating set $\mathcal{F}=\{\alpha_1,\dots,\alpha_b\}$ of the torsion-free part of $\Gamma$.
Finally, for $n\in\N$, we put $\mathcal{F}^{1/n}=\{\alpha_1^{1/n},\dots,\alpha_b^{1/n}\}$ and we define $\mu_n$ as the set of roots of unity killed by $n$.

Let $n$ be a positive integer dividing $m\in\N\cup\{\infty\}$ (by convention, all integers divide $\infty$). 
Galois theory claims that $\Gal(\Q(\mu_m, \mathcal{F}^{1/n})/\Q(\mathcal{F}))$ is isomorphic to the inner semidirect product of $\Gal(\Q(\mu_m, \mathcal{F})/\Q(\mathcal{F}))$ and $\Gal(\Q(\mu_m, \mathcal{F}^{1/n})/\Q(\mu_m, \mathcal{F}))$. 
Thus each element $\sigma\in\Gal(\Q(\mu_m, \mathcal{F}^{1/n})/\Q(\mathcal{F}))$ can be identified with a couple \[(\phi_\sigma, \psi_\sigma)\in \Gal(\Q(\mu_m, \mathcal{F})/\Q(\mathcal{F})) \times \Gal(\Q(\mu_m, \mathcal{F}^{1/n})/\Q(\mu_m, \mathcal{F})).\]
Concretely, if $x=\sum_{j=1}^l a_j \gamma_j\in \Q(\mu_m, \mathcal{F}^{1/n})$ with $a_j\in \Q(\mu_m, \mathcal{F})$ and $\gamma_j\in \langle \mathcal{F}^{1/n}\rangle$, then $\sigma x = \sum_{j=1}^l \phi_\sigma(a_j)\psi_\sigma(\gamma_j)$. 

The Cartesian product above can be explicitly described.
The computation of the left piece can be done thanks to the class field theory and that of the right piece by using a result of Perruca and Sgobba \cite[Theorem 13]{Perucca}, see the lemma below.

\begin{lmm} \label{lmm 16}
Let $L/\Q(\mathcal{F})$ be a finite extension. 
Then there exists an integer $C\in \N$, depending only on $\Gamma$ and $L$, such that for all $d_1,\dots,d_b\in\N$ divising $m$, we have \[ G=\Gal(L(\mu_m, \alpha_1^{1/d_1},\dots, \alpha_b^{1/d_b})/L(\mu_m))\simeq \prod_{l=1}^b \Z/(d_l/c_l)\Z\] for some positive integers $c_1,\dots,c_b\leq C$.
\end{lmm}

\begin{proof}
Put $L_0= L(\mu_m)$ and $L_l=L_{l-1}(\alpha_l^{1/d_l})$ for all $l\in\{1,\dots,b\}$. 
Let $G_l$ denote the Galois group of the extension $L_l/L_{l-1}$. 
Galois theory tells us that $G$ is the inner semidirect product of $G_b$ and $\Gal(L_{b-1}/L_0)$. 
As $G$ is abelian, this product is the Cartesian product. 
An easy induction shows that $G\simeq \prod_{l=1}^b G_l$. 
For all $l$, it is trivial that $G_l\simeq \Z/(d_l/c_l)\Z$ for some $c_l\in \N$. 
So $G\simeq \prod_{l=1}^b \Z/(d_l/c_l)\Z$. 
Write $d=\max\{d_1,\dots,d_b\}$. 
Theorem 13 in \cite{Perucca} claims that the extension $L_0(\mathcal{F}^{1/d})/L_0$ has degree at least $d^b/C$ for some $C\in\N$ depending only on $\Gamma$ and $L$.
On the other hand, $L_0(\mathcal{F}^{1/d})/L_0(\alpha_1^{1/d_1},\dots, \alpha_b^{1/d_b})$ has degree at most $d^b/\prod_{l=1}^b d_l$. 
The multiplicativity formula for degrees proves that $\#G \geq (\prod_{l=1}^b d_l)/C$; whence $\prod_{l=1}^b c_l\leq C$. 
\end{proof}

\begin{rqu} \label{rqu 2}
\rm{The isomorphism of Lemma \ref{lmm 16} is explicit: For each $(r_1,\dots,r_b)\in \prod_{l=1}^b \Z/(d_l/c_l)\Z$, there is an unique $\psi\in G$ such that $\psi \alpha_l^{1/d_l}=\zeta_{d_l/c_l}^{r_l} \alpha_l^{1/d_l}$ for all $l$.
Moreover, we have $\alpha_l^{1/c_l}\in L$ since $\psi\alpha_l^{1/c_l}=(\psi\alpha_l^{1/d_l})^{d_l/c_l}=\alpha_l^{1/c_l}$ for all $\psi\in G$. 
By abuse of notation, this isomorphism becomes from now an equality. }
\end{rqu}

\section{proof of Theorem \ref{thm0}} \label{section 1.5}
It is clear that Conjecture \ref{conj 0} $(iii)$ implies the assertion stated in Theorem \ref{thm0}. 
Now assume that this assertion is true and prove Conjecture \ref{conj 0} $(iii)$. 

Let $(x_n)$ be a sequence of $\Q(\Gamma_{\mathrm{div}})^*$ such that $h(x_n)\to 0$. 
We want to show that $x_n\in\Gamma_{\mathrm{div}}$ for all $n$ large enough. 
For this, assume by contradiction that $x_n\notin \Gamma_{\mathrm{div}}$ for infinitely many $n$, that is for all $n$ by taking a suitable subsequence. 

\begin{lmm} \label{lmm 0}
We have $[\Q(x_n^2) : \Q]\to +\infty$. 
\end{lmm}

\begin{proof}
Let $l$ be an accumulation point of the sequence $([\Q(x_n^2) : \Q])$.
Passing to a subsequence, we get $[\Q(x_n^2) : \Q]\to l$. 
If $l<+\infty$, then our sequence is bounded and Northcott's theorem implies that $x_n^2\in\mu_\infty\subset \Gamma_{\mathrm{div}}$ for all $n$ large enough, which is absurd. 
So $l=+\infty$ and the lemma follows.  
\end{proof}

Recall that $K_\Gamma=\Q(\mu_\infty, \Gamma)=\Q(\mu_\infty, \mathcal{F})$ and note that $\Q(\Gamma_{\mathrm{div}})=\bigcup_{n\in\N} K_\Gamma(\mathcal{F}^{1/n})$.

\begin{lmm} \label{lmm 0.5}
There is a sequence $(\sigma_n)$ of $\Gal(\overline{\Q}/\Q(\mathcal{F}))$ such that $d_{\Gamma, \varepsilon}(\sigma_n x_n)\underset{n\to+\infty}{\longrightarrow} 1$ for all $\varepsilon>0$. 
\end{lmm}

\begin{proof}
Let $n\in \N$.
 There exists $m_n\in\N$ such that $x_n\in K_\Gamma(\mathcal{F}^{1/m_n})$. 
We can also find a multiple $m'_n\in\N$ of $m_n$ such that: 
\begin{enumerate}[(i)]
\item $x_n\in \Q(\zeta_{m'_n}, \mathcal{F}^{1/m_n})$; 
\item The restriction map $\Gal(K_\Gamma(\mathcal{F}^{1/m_n})/K_\Gamma)\to H_n=\Gal(\Q(\zeta_{m'_n}, \mathcal{F}^{1/m_n})/\Q(\zeta_{m'_n}, \mathcal{F}))$ is an isomorphism. 
\end{enumerate}
Express $x_n$ as $x_n=\sum_{j=1}^{l_n} a_{j,n}\gamma_{j,n}$ with $a_{j,n}\in \Q(\zeta_{m'_n}, \mathcal{F})$ and $\gamma_{j,n}\in \langle \mathcal{F}^{1/m_n}\rangle$. 
Fix $\varepsilon>0$. 
For any $\phi\in N_n=\Gal(\Q(\zeta_{m'_n}, \mathcal{F})/\Q(\mathcal{F}))$, we set \[ d_{n}(\phi)=\frac{\#\left\{\psi\in H_n, 1-\varepsilon \leq \left\vert \sum_{j=1}^{l_n} \phi(a_{j,n})\psi(\gamma_{j,n})\right\vert^2\leq 1+\varepsilon \right\}}{\# H_n} \] and we pick $\phi_n\in N_n$ such that $d_n(\phi_n)=\max_{\phi\in N_n} \{d_n(\phi)\}$.

By Galois theory, $G_n=\Gal(\Q(\zeta_{m'_n}, \mathcal{F}^{1/m_n})/\Q(\mathcal{F}))$ is isomorphic to the inner semidirect product of $N_n$ and $H_n$. 
Denote by $\sigma_n$ the element of $G_n$ corresponding to the couple $(\phi_n, 1)\in N_n\times H_n$. 
Section \ref{section 1.25} tells us that $\sigma_n x_n = \sum_{j=1}^{l_n} \phi_n(a_{j,n})\gamma_{j,n}$. 

From (i)-(ii), we infer that $O_\Gamma(\sigma_n x_n)$ is equal to the orbit of $\sigma_n x_n$ under $H_n$. 
Thus it follows from the definition of $d_{\Gamma,\varepsilon}(\sigma_n x_n)$ (see Subsection \ref{subsection 1.2}) that \[d_{\Gamma, \varepsilon}(\sigma_n x_n) = \frac{\#\{\psi \in H_n, 1-\varepsilon\leq \vert \psi\sigma_n x_n\vert^2\leq 1+\varepsilon\}}{\#H_n} = d_n(\phi_n). \]
A small calculation gives \[ \begin{aligned}
u_n &= \#\{\sigma\in G_n,  1-\varepsilon\leq \vert \sigma x_n\vert^2\leq 1+\varepsilon \} \\ 
&= \#\left\{(\phi,\psi)\in N_n\times H_n, 1-\varepsilon\leq  \left\vert \sum_{j=1}^{l_n} \phi(a_{j,n})\psi(\gamma_{j,n})\right\vert^2\leq 1+\varepsilon\right\} \\ 
&= \sum_{\phi\in N_n} \#\left\{\psi \in H_n, 1-\varepsilon\leq  \left\vert \sum_{j=1}^{l_n} \phi(a_{j,n})\psi(\gamma_{j,n})\right\vert^2\leq 1+\varepsilon\right\},
\end{aligned}\] and so
\begin{equation} \label{eq0}
 \frac{u_n}{\#G_n} = \frac{1}{\#N_n}\sum_{\phi\in N_n} d_n(\phi) \leq d_n(\phi_n)= d_{\Gamma,\varepsilon}(\sigma_n x_n)\leq 1. 
\end{equation}
Recall that $\mathcal{F}$ is a finite set. 
As $h(x_n^2)=2h(x_n)\to 0$ and $[\Q(\mathcal{F}, x_n^2) :\Q(\mathcal{F})]= [\Q(\mathcal{F}, x_n^2) : \Q]/[\Q(\mathcal{F}):\Q]\to +\infty$ by Lemma \ref{lmm 0}, we infer from Bilu's equidistribution Theorem \cite[Subsection 1.1, Theorem]{Yuan} applied to $K=\Q(\mathcal{F})$ that $u_n/\#G_n$ goes to $1$ as $n\to +\infty$. 
Lemma follows by involving the squeeze theorem in \eqref{eq0}. 
\end{proof}

\textit{Proof of Theorem 1.4:}
We have $h(\sigma_n x_n)=h(x_n)\to 0$ and $d_{\Gamma, \varepsilon}(\sigma_n x_n)\underset{n\to+\infty}{\longrightarrow} 1$ for all $\varepsilon>0$ by Lemma \ref{lmm 0.5}.
Applying the assertion of Theorem \ref{thm0} (which is assumed to be true) to $y_n=\sigma_n x_n$ shows that $\sigma_n x_n\in\Gamma_{\mathrm{div}}$ for all $n$ large enough. 

By definition of the division group, we deduce that $\sigma_n (x_n^{j_n})\in \langle \mathcal{F}\rangle $ for some integer $j_n\in\N$. 
As $\sigma_n$ fixes the elements of $\mathcal{F}$, we conclude $x_n^{j_n}\in \langle \mathcal{F} \rangle\subset \Gamma$, which contradicts the fact that $x_n\notin \Gamma_{\mathrm{div}}$. 
This finishes the proof of Theorem \ref{thm0}. 
\qed

\section{Direct image of $\mu_d$ under a meromorphic function} 

\begin{defn} \label{defn 2}
\rm{ Let $N\in \N$. 
Denote by $\mathcal{S}_N$ the set of integers $d\in\N$ for which there exists a meromorphic function $f(z)= \sum_{j=1}^N a_j z^{b_j}$ on $\C^*$ satisfying $\vert f(\zeta) \vert=1$ for all $\zeta\in \mu_d$. 
Moreover, we impose that: 
\begin{itemize}
\item $b_1=0,b_2, \dots,b_N\in \Z$ are pairwise distinct integers such that $b_1,\dots,b_N$ and $d$ have no common positive factors other than $1$; 
\item $a_1,\dots,a_N\in \C$ satisfy $\sum_{i\in I} a_i\neq 0$ for all non-empty subsets $I\subset \{1,\dots, N\}$.
\end{itemize}}
\end{defn}

The study of $\mathcal{S}_N$ might be of independent interest, but we only prove in this section what is needed for this article, namely the finiteness of it for all $N\in\N$. 

Let $N\geq 2$, and let $d\in\mathcal{S}_N$ be greater than $4^N (N^2-1)$. 
The lemma below asserts that regardless of the choice of integers $b_1,\dots,b_N$ as in Definition \ref{defn 2}, at least one of them has an absolute value greater than or equal to $d/4$. 
We will then contradict this claim thanks to Dirichlet's Theorem on simultaneous approximation.
                                                                                                                                                                                                   
\begin{lmm} \label{lmm Fischler}
Let $M\geq 2$, and let $g(z)=\sum_{j=1}^M u_j z^{c_j}$ be a meromorphic function on $\C^*$ with $u_1,\dots,u_M\in\C^*$ and $c_1,\dots,c_M\in\Z$ pairwise distinct integers.
If $d\geq M^2$ is an integer such that $\vert g(\zeta)\vert =1$ for all $\zeta\in \mu_d$, then $\max\{\vert c_1\vert ,\dots, \vert c_M\vert \} \geq d/4$.
\end{lmm}

\begin{proof} 
Assume by contradiction that $\max\{\vert c_1\vert ,\dots, \vert c_M\vert \} < d/4$.
Write \[\{\gamma_1,\dots, \gamma_n\}=\{c_i-c_j, 1\leq i,j \leq M\}\] with $n\leq M^2$ and $\gamma_1,\dots, \gamma_n$ pairwise distinct.
Note that $0=c_1-c_1\in\{\gamma_1,\dots,\gamma_n\}$. 
We can thus assume that $\gamma_1=0$.
Moreover, $\gamma_k\in ]-d/2, d/2[$ for all $k$ since $\max\{\vert c_1\vert ,\dots, \vert c_M\vert \} < d/4$.

For each $k\in\{1,\dots,n\}$, define $E_k$ as the set of couples $(i,j)\in\{1,\dots, M\}^2$ for which $\gamma_k=c_i-c_j$.
Clearly $E_1,\dots, E_n$ is a partition of $\{1,\dots,M\}^2$. 
Then put $v_k= \sum_{(i,j)\in E_k} u_i\overline{u_j}$. 
An easy calculation gives, for all $l\in\{0,\dots,d-1\}$,  
\begin{equation} \label{eq system}
\begin{aligned}
\sum_{k=1}^n \zeta_d^{l\gamma_k}  v_k & = \sum_{k=1}^n \sum_{(i,j)\in E_k}u_i\overline{u_j} \zeta_d^{l(c_i-c_j)} \\
& = \sum_{(i,j)\in \bigsqcup_{k=1}^n E_k} u_i\overline{u_j} \zeta_d^{l(c_i-c_j)} =  \sum_{1\leq i,j\leq M} u_i\overline{u_j} \zeta_d^{l(c_i-c_j)} \\ 
& = \left(\sum_{i=1}^M u_i \zeta_d^{lc_i}\right) \left(\sum_{j=1}^M \overline{u_j} \zeta_d^{-l c_j}\right) = g(\zeta_d^l)\overline{g(\zeta_d^l)}= 1,
\end{aligned}
\end{equation}
 the last equality coming from the fact that $\vert g(\zeta)\vert=1$ for all $\zeta\in\mu_d$.
 We recognize a linear system with $n$ unknowns (namely $v_1,\dots,v_n$) and $d$ equations. 
By assumption, $d\geq M^2\geq n$.
The (square) matrix associated to the linear subsystem
\begin{equation} \label{eq subsystem}
  \forall l\in\{0,\dots,n-1\}, \; \; \; \; \; \sum_{k=1}^n \zeta_d^{l\gamma_k}  v_k = 1
\end{equation} 
is the Vandermonde matrix $\left(\zeta_d^{l \gamma_k}\right)_{\underset{1\leq k \leq n}{0\leq l\leq n-1}}$.
It is well known that its determinant is $\prod_{1\leq i<j \leq n} (\zeta_d^{\gamma_j}-\zeta_d^{\gamma_i})$.
As $\gamma_1,\dots,\gamma_n\in ]-d/2, d/2[$ are pairwise distinct, the numbers $\zeta_d^{\gamma_1},\dots, \zeta_d^{\gamma_n}$ are therefore pairwise distinct. 
Hence the determinant above is non-zero, and so \eqref{eq subsystem} has a unique solution. 
Thus \eqref{eq system} has at most one solution. 
As $\gamma_1=0$, it follows that $(v_1,\dots,v_n)=(1,0,\dots,0)$ is the unique solution of \eqref{eq system}.

Denote by $i_0$, resp. $j_0$, the element of $\{1,\dots, M\}$ such that 
\begin{equation} \label{eq max moins min}
c_{i_0}=\max\{c_1,\dots,c_M\}, \; \text{resp.} \; c_{j_0}=\min\{c_1,\dots,c_M\}.
\end{equation}
Let $m\in\{1,\dots,n\}$ be the integer such that $\gamma_m=c_{i_0}-c_{j_0}$, and let $(i,j)\in E_m$. 
Then $c_i-c_j=c_{i_0}-c_{j_0}$ and \eqref{eq max moins min} leads to $c_i=c_{i_0}$ and $c_j=c_{j_0}$. 
As $c_1,\dots,c_M$ are pairwise distinct, we get $(i,j)=(i_0,j_0)$.  
In conclusion, $E_m=\{(i_0, j_0)\}$. 

Recall that $\gamma_1=0$. 
So the set $E_1$ contains at least $M\geq 2$ elements, namely $(1,1),\dots,(M,M)$; whence $m >1$. 
But then, by the foregoing, we have \[0= v_m=\sum_{(i,j)\in E_m} u_i \overline{u_j} = u_{i_0}\overline{u_{j_0}},\] i.e. either $u_{i_0}=0$ or $u_{j_0}=0$, which is absurd. 
This completes the proof.
\end{proof}

We can now prove the desired result. 

\begin{thm} \label{thm S N fini}
For all $N\in \N$, the set $\mathcal{S}_N$ is finite. 
Moreover, $\mathcal{S}_1=\{1\}$ and $4^N (N^2-1)$ is an upper bound of $\mathcal{S}_N$ for all $N\geq 2$.
\end{thm}

\begin{proof}
Let $d\in \mathcal{S}_N$.
There is a meromorphic function $f(z)= \sum_{j=1}^N a_jz^{b_j}$ on $\C^*$ such that $\vert f(\zeta)\vert=1$ for all $\zeta\in\mu_d$, where $a_1,\dots, a_N$ and $b_1,\dots, b_N$ are as in Definition \ref{defn 2}.
If $N=1$, then $d$ has to be coprime to $b_1=0$, i.e. $d=1$. 
As $\vert 1\vert=1$, we get $\mathcal{S}_1=\{1\}$. 
Now assume that $N\geq 2$. 

Dirichlet's theorem on simultaneous approximation asserts the existence of integers $p_1,\dots,p_N\in\Z$ and $q\in\{1,\dots,4^N\}$ such that $\vert q (b_j/d) - p_j\vert < 1/4$ for all $j\in\{1,\dots,N\}$.
Write $e$ for the greatest common divisor of $q$ and $d$, then put $q'=q/e$ and $d'=d/e$. 
Clearly $q'$ and $d'$ are coprime, $e\in\{1,\dots,4^N\}$ and 
\begin{equation} \label{eq encore une et pourquoi elle}
\forall j\in\{1,\dots, N\}, \; \; \; \; \; \vert q' b_j - d'p_j\vert < d'/4.
\end{equation}
Put $\{c_1,\dots,c_M\}=\{q'b_1-d'p_1, \dots, q'b_N-d'p_N\}$, where $c_1,\dots,c_M$ are pairwise distinct. 
We can assume that $c_1=q'b_1-d'p_1$.
For all $k\in\{1,\dots, M\}$, denote by $E_k$ the set of integers $j\in\{1,\dots,N\}$ such that $c_k=q'b_j-d'p_j$.
Clearly $E_1,\dots, E_M$ is a partition of $\{1,\dots,N\}$. 
Put $u_k=\sum_{j\in E_k} a_j$ and $g(z)=\sum_{k=1}^M u_k z^{c_k}$. 

As $b_1=0$ by Definition \ref{defn 2} and $\vert p_1\vert<1/4$ by \eqref{eq encore une et pourquoi elle}, we conclude $c_1=0$.  
Let $n$ be a common positive factor of $c_1,\dots,c_M$ and $d'$. 
From the definition of $c_i$, we deduce that $n$ divides $q'b_1,\dots,q'b_N$ and $q'd'$.
By Definition \ref{defn 2}, $b_1,\dots,b_N$ and $d'$ have no common positive factors other than $1$ since $d'$ divides $d$. 
Hence $n$ divides $q'$. 
But it divides $d'$ too. 
Whence $n=1$ since $q'$ and $d'$ are coprime.

We have $\sum_{k\in I} u_k \neq 0$ for all non-empty subsets $I\subset \{1,\dots, M\}$ since otherwise, \[0= \sum_{k\in I} u_k = \sum_{k\in I} \sum_{j\in E_k} a_j = \sum_{j\in \bigsqcup_{k\in I} E_k} a_j,\] which disagrees with Definition \ref{defn 2}. 
In particular, $u_1,\dots,u_M\in \C^*$.
Finally,
\begin{align*}
\forall l\in\Z, \; \; \; \; \;  f(\zeta_d^{ql}) & = \sum_{j=1}^N a_j\zeta_d^{q b_j l} =  \sum_{j=1}^N a_j\zeta_{d'}^{q'b_j l} = \sum_{j=1}^N a_j\zeta_{d'}^{(q'b_j-d'p_j)l} \\
& = \sum_{k=1}^M\sum_{j\in E_k} a_j\zeta_{d'}^{(q'b_j-d'p_j)l} = \sum_{k=1}^M u_k \zeta_{d'}^{c_kl} = g(\zeta_{d'}^l).
\end{align*}
As $\vert f(\zeta)\vert=1$ for all $\zeta\in \mu_d$, we infer that $\vert g(\zeta)\vert=1$ for all $\zeta\in \mu_{d'}$. 

Of all this, we conclude $d'\in \mathcal{S}_M$. 
Recall that $d'=d/e$ with $e\in\{1,\dots, 4^N\}$, that $u_1,\dots,u_M\in\C^*$, that $c_1,\dots,c_M\in\Z$ are pairwise distinct, that $\vert g(\zeta)\vert=1$ for all $\zeta\in \mu_{d'}$ and that $\max\{\vert c_1\vert, \dots, \vert c_M\vert\}<d'/4$ by \eqref{eq encore une et pourquoi elle}.
If $M\geq 2$, then Lemma \ref{lmm Fischler} applied to $d=d'$ implies $d'=d/e < M^2$; whence $d \leq 4^N (N^2-1)$.
If $M=1$, then $d'=1$ since $\mathcal{S}_1=\{1\}$.
Thus $d\leq 4^N$ and the theorem follows.
\end{proof}

\section{Proof of Theorem \ref{thm 1}} \label{section X}
Fix once and for all an integer $N\in\N$.
Recall that the length of an element $y\in\Q(\Gamma_{\mathrm{div}})$ is defined to be the smallest $l\in\N$ for which $y$ can express as $y=\sum_{j=1}^l y_j \gamma_j$ with $y_j\in K_\Gamma=\Q(\mu_\infty, \Gamma)$ and $\gamma_j\in\Gamma_{\mathrm{div}}$. 
The lemma below shows that the length is invariant under translation of points in $\Gamma_{\mathrm{div}}$. 

\begin{lmm} \label{lmm meme longiueur}
Let $y\in \Q(\Gamma_{\mathrm{div}})$, and let $z\in \Gamma_{\mathrm{div}}$. 
Then $y$ and $yz$ have the same length. 
\end{lmm}

\begin{proof}
Denote by $l$, resp. $l'$, the length of $y$, resp. $yz$. 
We can express $y$ as $y=\sum_{j=1}^l y_j \gamma_j$ with $y_j\in K_\Gamma$ and $\gamma_j\in\Gamma_{\mathrm{div}}$. 
Hence $yz=\sum_{j=1}^l y_j (\gamma_j z)$ and the definition of the length leads to $l'\leq l$. 
The inequality $l\leq l'$ is obtained by replacing $y$ with $yz$ and $z$ with $z^{-1}\in\Gamma_\mathrm{div}$.
This completes the proof.
\end{proof}

Recall that $\mathcal{F}=\{\alpha_1,\dots,\alpha_b\}$ is the torsion-free part of $\Gamma$ and that $l_N(\Gamma)$ denotes the set of elements in $\Q(\Gamma_{\mathrm{div}})$ with length $\leq N$. 
Concretely, each $x\in l_N(\Gamma)\backslash \{0\}$ with length $l$ can express as $x=\sum_{j=1}^l x_j \prod_{t=1}^b \alpha_t^{k_{j,t}/d}$, where $x_j\in K_\Gamma^*, d\in\N$ and $k_{j,t}\in\Z$.
We also recall that $U(\Gamma)$ is the set of elements in $\overline{\Q}^*$ for which all its conjugates over $K_\Gamma$ are concyclic and located on a circle centered at the origin.  
Finally, recall that for any algebraic set $X$, we define $X.\Gamma_{\mathrm{div}}$ as the set of $x\gamma$ with $x\in X$ and $\gamma\in\Gamma$.

\begin{prop} \label{lmm LL}
There is $m\in\N$ such that $l_N(\Gamma)\cap U(\Gamma)\subset K_\Gamma(\mathcal{F}^{1/m})^*.\Gamma_\mathrm{div}$.
\end{prop}

\begin{proof}
Let $x\in l_N(\Gamma)\cap U(\Gamma)$ be an element of lenght $l\leq N$ that we express as above.
For each $t\in\{1,\dots,b\}$, set $j_t\in\{1,\dots,l\}$ such that $k_{j_t,t}=\min\{k_{1,t},\dots, k_{l,t}\}$.
Let $D_t$ denote the greatest common divisor of $k_{1,t}-k_{j_t,t}, \dots, k_{l,t}-k_{j_t,t}$ and $d$, then $c_{j,t}= (k_{j,t}-k_{j_t,t})/D_t$ and $d_t=d/D_t$.
By construction, $c_{1,t},\dots,c_{l,t}$ and $d_t$ have no common positive factors other than $1$ for all $t\in\{1,\dots,r\}$. 
Finally, set
\begin{equation} \label{eq derniere}
y= \sum_{j=1}^l x_j \prod_{t=1}^b \alpha_t^{c_{j,t}/d_t} \; \; \text{and} \; \; z=\prod_{t=1}^b \alpha_t^{k_{j_t,t}/d}\in \Gamma_\mathrm{div}.
\end{equation} 
We easily check that $x=yz$. 
As $z\in\Gamma_{\mathrm{div}}$, Lemma \ref{lmm meme longiueur} tells us that $y$ has length $l$ too. 
Let $s\in\{1,\dots,b\}$ be such that $d_s=\max\{d_1,\dots,d_b\}$.
If $d_s$ is bounded from above by a constant $c$ depending only on $\Gamma$ and $N$, then \eqref{eq derniere} would show that $y\in K_\Gamma(\mathcal{F}^{1/c!})$. 
The proposition with $m=c !$ would follow since $x=yz$ with $z\in\Gamma_{\mathrm{div}}$. 

Put $\{b_1,\dots,b_n\}=\{c_{1,s}, \dots, c_{l,s}\}$, where $b_1,\dots,b_n$ are pairwise distinct. 
We can assume that $b_1=c_{j_s,s}$.
For all $k\in\{1,\dots, n\}$, write $E_k$ the set of $j\in\{1,\dots,l\}$ such that $b_k=c_{j,s}$.
Clearly $E_1,\dots, E_n$ is a partition of $\{1,\dots,l\}$. 

Lemma \ref{lmm 16} applied to $L=\Q(\Gamma)$ and $m=\infty$ as well as the multiplicativity formula for degrees claim that the extension \[L/M= K_\Gamma(\alpha_1^{1/d_1},\dots, \alpha_r^{1/d_r}) / K_\Gamma(\alpha_1^{1/d_1},\dots, \alpha_{s-1}^{1/d_{s-1}}, \alpha_{s+1}^{1/d_{s+1}}, \dots, \alpha_r^{1/d_r})\] has degree at least $d_s/C$ for some $C>0$ depending only on $\Gamma$. 
Its Galois group is therefore isomorphic to $\Z/(d_s/c)\Z$ for some $c\leq C$. 
Put \[y_j= x_j \prod_{t=1, t\neq s}^b \alpha_t^{c_{j,t}/d_t}\in M, v_k=\sum_{j\in E_k} y_j, a_k=v_k\alpha_s^{b_k/d_s} \; \text{and} \; f(z)=\sum_{k=1}^n (a_k/y) z^{b_k}.\]  
Note that $y = \sum_{j=1}^l y_j \alpha_s^{c_{j,s}/d_s}$.  
Establish below that $d_s/c\in \mathcal{S}_n$ (see Definition \ref{defn 2}). 

First $c_{j_s,s}=0$ by definition of $c_{j,t}$; whence $b_1=0$. 
Next $b_1,\dots,b_n$ and $d_s/c$ have no common positive factors other than $1$ by construction of $c_{1,s}, \dots, c_{l,s}$ and $d_s$. 

Assume that $\sum_{k\in I}a_k=0$ for some subset $I\subset \{1,\dots,n\}$.
Then \[ \sum_{j\in \bigsqcup_{k\in I} E_k}  y_j \alpha_s^{c_{j,s}/d_s}=\sum_{k\in I}\sum_{j\in E_k}y_j\alpha_s^{c_{j,s}/d_s}=\sum_{k\in I} v_k\alpha_s^{b_k/d_s}=\sum_{k\in I} a_k= 0,\] which allows us to deduce that \[y=\sum_{j=1}^l y_j \alpha_s^{c_{j,s}/d_s} = \sum_{j\notin \bigsqcup_{k\in I} E_k} y_j\alpha_s^{c_{j,s}/d_s} = \sum_{j\notin \bigsqcup_{k\in I} E_k} x_j\prod_{t=1}^b \alpha_t^{c_{j,t}/d_t}.\] 
The definition of the length proves that $y$ has length at most $l-\#  \bigsqcup_{k\in I} E_k$. 
As $y$ has length $l$ and $E_1,\dots,E_n$ are non-empty sets, we conclude that $I$ is empty. 
The contrapositive proves that $\sum_{k\in I} a_k \neq 0$ for all non-empty subsets $I\subset \{1,\dots, n\}$.

Let $\sigma\in \Gal(L/M)$. 
As $y_j\in M$ for all $j$, we get $v_k\in M$ for all $k$. 
Collecting the information above, we obtain 
\begin{equation} \label{eq 1234}
\begin{aligned}
\sigma y & = \sigma\left(\sum_{j=1}^l y_j \alpha_s^{c_{j,s}/d_s}\right) = \sigma\left(\sum_{k=1}^n \sum_{j\in E_k} y_j \alpha_s^{c_{j,s}/d_s}\right) \\ 
& =  \sigma\left(\sum_{k=1}^n v_k \alpha_s^{b_k/d_s}\right)= \sum_{k=1}^n v_k \alpha_s^{b_k/d_s} \zeta_{d_s/c}^{b_k l_\sigma}=y f(\zeta_{d_s/c}^{l_\sigma})
\end{aligned}
\end{equation}
for some $l_\sigma\in\{1,\dots, d_s/c\}$.
Recall that $y=x/z$, that $x\in U(\Gamma)$ and that $z\in\Gamma_{\mathrm{div}}\subset U(\Gamma)$ (see the introduction). 
As $U(\Gamma)$ is a group, we deduce that $y\in U(\Gamma)$, and so $\vert \sigma y\vert = \vert y \vert$.
By varying $\sigma\in \Gal(L/M)$, we infer thanks to \eqref{eq 1234} that $\vert f(\zeta)\vert = 1$ for all $\zeta\in \mu_{d_s/c}$. 
From all this, we finally conclude $d_s/c\in \mathcal{S}_n$. 
By Theorem \ref{thm S N fini}, we get $d_s/c \leq 4^n n^2$, i.e. $d_s\leq 4^N N^2 C$, which ends the proof of the proposition. 
\end{proof}

\textit{Proof of Theorem \ref{thm 1}.} By Proposition \ref{lmm LL}, small points of $l_N(\Gamma)\cap U(\Gamma)$ lie in $K_\Gamma(\mathcal{F}^{1/m})^*.\Gamma_\mathrm{div}$ for some $m\in\N$. 
Theorem \ref{thm 1} follows by applying Lemma \ref{lmm 20} to $F=K_\Gamma(\mathcal{F}^{1/m})$.
\qed

\section{Equidistribution} \label{section 2}
As stated in the introduction, we will use equidistribution arguments to prove Theorem \ref{thm 2}, and especially Corollary \ref{lmm 9}. 
This section is therefore devoted to the proof of this corollary. 

In this section, $M\in \N$ denotes a positive integer.
The Weil height can extend to $\overline{\Q}^M$, see \cite[Section 1.5]{BombieriGubler}; call it $h$ again. 
We have $h(\boldsymbol{\zeta})=0$ for all $\boldsymbol{\zeta}\in\mu_\infty^M$. 

We say that a sequence $(\mu_n)$ of probability measures on $S=(\C^*)^M$ weakly converges to $\mu$, denoted by $\mu_n \xrightarrow{w} \mu$, if $\int_S f d\mu_n \to \int_S f d\mu$ for any bounded continuous function $f : S \to \R$.
For a finite set $F\subset S$, the discrete probability measure on $S$ associated to it is given by \[ \mu_F=\frac{1}{\#F} \sum_{\alpha\in F} \delta_\alpha, \] where $\delta_\alpha$ is the Dirac measure on $S$ supported on $\alpha$.

Write $\nu_M$ for the uniform probability measure on $S$ supported at the unit polycircle $\vert z_1\vert = \dots = \vert z_M\vert =1$, where it coincides with the normalized Haar measure.

Finally, we say that a sequence $(P_n)$ of $(\overline{\Q}^*)^M$ is strict if any proper algebraic subgroup of $(\overline{\Q}^*)^M$ contains $P_n$ for only finitely many $n$. 
When $M=1$, it is equivalent to saying that $[\Q(P_n) : \Q] \to +\infty$, see \cite[Lemma 5.2.1]{Petsche}.

The arguments presented by Bilu in \cite{Bilu} are quite general and can easily be adapted to produce other equidistribution results. 
We see an example of this below. 

\begin{prop} \label{prop 1}
Let $((\zeta_{m_n}^{k_{1,n}},\dots,\zeta_{m_n}^{k_{M,n}}))$ be a strict sequence with $m_n\in\N$ and $k_{1,n},\dots,k_{M,n}\in\Z$ for all $n$. 
Then $\mu_{E_n} \xrightarrow{w} \nu_M$ where \[E_n = \{(\zeta_{m_n}^{sk_{1,n}},\dots,\zeta_{m_n}^{sk_{M,n}}), s=1,\dots,m_n\}.\]
\end{prop}
\begin{proof}
First of all, we alert the reader on the fact that the hypothesis "$F_n$ are pairwise distinct" made in \cite[Th\'eor\`eme 2]{FavreLetelier} can be weakened in $\#F_n \to +\infty$, the former assumption being only used to get the latter one, see \cite[Sous-section 5.4]{FavreLetelier}.

Let $\mathbf{n}=(n_1,\dots,n_M)\in \Z^M\backslash \{(0,\dots,0)\}$, and set $\chi_{\mathbf{n}} : S \to \C$ to be the function defined by $\chi_{\mathbf{n}}(z_1,\dots,z_M)=\prod_{i=1}^M z_i^{n_i}$. 
We clearly have \[\chi_\mathbf{n}(E_n) = \left\{ \zeta_{m_n}^{s\sum_{i=1}^M n_i k_{i,n}}, s=1,\dots,m_n \right\}.\] 
The sequence $(\zeta_{m_n}^{\sum_{i=1}^M n_i k_{i,n}})$ is strict since $((\zeta_{m_n}^{k_{1,n}},\dots,\zeta_{m_n}^{k_{M,n}}))$ is by hypothesis. 
Thus \[\# \chi_\mathbf{n}(E_n) \geq \left[ \Q\left( \zeta_{m_n}^{\sum_{i=1}^M n_i k_{i,n}} \right) : \Q \right] \to +\infty.\] 
Thanks to \cite[Th\'eor\`eme 2]{FavreLetelier}, we get $\mu_{\chi_\mathbf{n}(E_n)} \xrightarrow{w} \nu_1$. 
From this observation, and following the lines of the proof of \cite[Proposition 4.1]{Bilu}, we conclude that \cite[Proposition 4.1]{Bilu} also holds for the sequence $(\mu_{E_n})$. 

For $z=(z_1,\dots,z_M)\in S$, put $\vert z\vert_\infty= \max\{\vert z_1\vert, \dots, \vert z_M\vert\}$. 
Let $\varepsilon>0$ and write \[K_\varepsilon=\{ z\in S, \; \max\{ \vert z\vert_\infty, \vert z\vert_\infty^{-1}\} \leq e^{2/\varepsilon} \}\]
Let $A$ be the family of sets $E\subset (\overline{\Q}^*)^M$ that are finite, $\Gal(\overline{\Q}/\Q)$-invariant and such that $h(\boldsymbol{\beta}),h(\boldsymbol{\beta}^{-1}) \leq 1$ for all $\boldsymbol{\beta}\in E$. 
We clearly have $E_n\in A$ for all $n$.

Choose $E\in A$. 
As $E$ is both finite and $\Gal(\overline{\Q}/\Q)$-invariant, we can decompose it as a finite disjoint union of Galois orbits $\Gal(\overline{\Q}/\Q).\boldsymbol{\beta}_1, \dots, \Gal(\overline{\Q}/\Q).\boldsymbol{\beta}_k$ for some $\boldsymbol{\beta}_1,\dots,\boldsymbol{\beta}_k\in E$.  
For each $i\in\{1,\dots,k\}$, Bilu proved in \cite[Section 4]{Bilu} that $\sigma \boldsymbol{\beta}_i \notin K_\varepsilon$ for at most $\varepsilon[\Q(\boldsymbol{\beta}_i) : \Q]$ field embeddings $\sigma : \Q(\boldsymbol{\beta}_i) \to \C$. 
Thus $E\backslash K_\varepsilon$ has at most $\varepsilon\sum_{i=1}^k [\Q(\boldsymbol{\beta}_i) : \Q]=\varepsilon \#E$ elements, i.e. $\mu_E(K_\varepsilon) \geq 1-\varepsilon$.
It remains to mimic the proof of \cite[Theorem 1.1]{Bilu} made in \cite[Section 4]{Bilu} to deduce the proposition.
\end{proof}

For $\zeta, \zeta'$ belonging to the unit circle in $\C$, we denote by $[\zeta, \zeta']$ the arc of this unit circle connecting $\zeta$ and $\zeta'$ anticlockwise.  

\begin{cor} \label{lmm 9}
Let $((\zeta_{m_n}^{k_{1,n}}, \dots,\zeta_{m_n}^{k_{M,n}}))$ be a strict sequence, with $m_n\in\N$ and $k_{1,n}, \dots, k_{M,n} \in\Z$ for all $n$. 
Let $\varepsilon >0$, and let $x_1,\dots,x_M\in [0; 2\pi[$. 
Write $V=\prod_{j=1}^M [ e^{\zeta_4(x_j-\varepsilon)}; e^{\zeta_4(x_j+\varepsilon)}]$ and denote by $\mathcal{K}_n(V)$ the set of $r\in \Z/m_n\Z$ such that $(\zeta_{m_n}^{r k_{1,n}}, \dots, \zeta_{m_n}^{r k_{M,n}}) \in V$. 
Then, for all $n$ large enough, we have \[ \frac{\# \mathcal{K}_n (V)}{m_n} \geq (1-\varepsilon) \left(\frac{\varepsilon}{2\pi}\right)^M. \]
\end{cor}
\begin{proof}
 Let $f : \C^M \to [0;1]$ be any continuous function that is identically zero outside $V$ and taking the value $1$ on \[W =\prod_{j=1}^M [ e^{\zeta_4(x_j-\varepsilon/2)}; e^{\zeta_4 (x_j+\varepsilon/2)}]\subsetneq V.\]
Let $n\in\N$. 
For $r\in \Z$, we set \[P_{r,n} = \left(\zeta_{m_n}^{r k_{1,n}}, \dots, \zeta_{m_n}^{r k_{M,n}}\right) \; \text{and} \; E_n = \{ P_{r,n}, r=1,\dots, m_n \}. \] 
Let $r_n$ be the smallest divisor of $m_n$ such that $r_n k_{j,n} \equiv 0 \; (m_n)$ for all $j\in \{1,\dots,M\}$. 
Thanks to the equality $P_{r+r_n,n}=P_{r,n}$, valid for all $r\in\Z$, we infer that $E_n = \{ P_{r,n}, r=1,\dots,r_n \}$. 
Furthermore, the minimality of $r_n$ gives $\# E_n = r_n$.
Thus \[\frac{1}{ m_n} \sum_{r=1}^{m_n} \delta_{P_{r,n}} = \frac{1}{r_n}\sum_{r=1}^{r_n} \delta_{P_{r,n}} = \mu_{E_n}. \]
Combining this equality with Proposition \ref{prop 1} provides the limit \[u_n= \frac{1}{m_n}\sum_{r=1}^{m_n} f(P_{r,n})  =\frac{1}{\#E_n}\sum_{\alpha\in E_n} f(\alpha) \to \int f d\nu_M\] since the sequence $(P_{1,n})$ is strict by assumption.
Thus $u_n\geq (1-\varepsilon)\int f d\nu_M$ for all $n$ large enough. 
The lemma follows by noticing that the construction of $f$ implies the inequalities $u_n \leq \frac{1}{m_n}\#\mathcal{K}_n(V)$ and $\int f d\nu_M \geq \nu_M(W)=(\varepsilon/2\pi)^M$. 
\end{proof}

\section{A crucial subsequence} \label{section 4}
To prove our Theorem \ref{thm 2}, we need to extract from $(x_n)$ a "good" subsequence whose construction is the aim of this section. 

Recall that we fixed an integer $N\in\N$ and a generating set $\{\alpha_1,\dots,\alpha_b\}$ of the torsion-free part of $\Gamma$. 
Let $(x_n)$ be a sequence of $l_N(\Gamma)$ (see the introduction for a definition).
Each term can express as \[x_n= \sum_{j=1}^N x_{j,n}\prod_{l=1}^b \alpha_l^{k_{j,l,n}/m_n}\]
with $x_{j,n}\in K_\Gamma=\Q(\mu_\infty, \Gamma), m_n \in\N$ and $k_{j,l,n}\in \Z$. 

By convention, a sum indexed by the empty set is always $0$. 

\begin{lmm} \label{cor 3}
There exists a subsequence $(x_{\psi(n)})$ of $(x_n)$ satisfying the following: for all $l\in\{1,\dots,b\}$, there is a set $J_l\subset \{1,\dots,N\}$ such that 
\begin{enumerate} [a)]
\item the sequence of terms $(\zeta_{m_{\psi(n)}}^{k_{j,l,\psi(n)}})_{j\in J_l}$ is strict unless $J_l$ is empty; 
\item for all $j\in\{1,\dots,N\}$, there are an integer $\lambda^{(j,l)}\in \Z\backslash \{0\}$ and a tuple $(\lambda_m^{(j,l)})_{m\in J_l}\in \Z^{\#J_l}$ such that for all $n$,  \[ \lambda^{(j,l)}k_{j,l,\psi(n)}+\sum_{m\in J_l} \lambda_m^{(j,l)} k_{m,l,\psi(n)}\equiv 0 \;  (m_{\psi(n)}). \] 
\end{enumerate} 
\end{lmm}
\begin{proof}
We compare the elements in $\R^2$ with the lexicographical order $\preceq$. 

Construct recursively sets $J_{l,t}\subset \{1,\dots,N\}$ and functions $\psi_{l,t} : \N\to\N$, where $(l,t)$ ranges over all elements of $I=\{1,\dots,b\} \times \{0,\dots,N\}$, as follows: 
If $t=0$, then $J_{l,0}$ is empty and $\psi_{l,0}$ is either the identity if $l=1$ or $\psi_{l-1,N}$ if $l>1$. 
Assume that $t\geq 1$. 
If the sequence of terms $\left(\zeta_{m_{\psi_{l,t-1}(n)}}^{k_{j,l,\psi_{l,t-1}(n)}}\right)_{j\in J_{l,t-1} \cup \{t\}}$ is strict, then we put $J_{l,t}=J_{l,t-1}\cup \{t\}$ and $\psi_{l,t}=\psi_{l,t-1}$. 
If not, then put $J_{l,t}=J_{l,t-1}$. 
By definition of a strict sequence, there are a proper algebraic subgroup $T_{l,t}$ of $\overline{\Q}^*$ and a subsequence $(x_{\psi_{l,t}(n)})$ of $(x_{\psi_{l,t-1}(n)})$ such that $u_{l,t,n}=\left(\zeta_{m_{\psi_{l,t}(n)}}^{k_{j,l,\psi_{l,t}(n)}}\right)_{j\in J_{l,t} \cup \{t\}}\in T_{l,t}$ for all $n$. 

From this construction, we easily check by induction that for all $(l,t)\in I$, either $J_{l,t}$ is empty or the sequence of terms $v_{l,t,n}=\left(\zeta_{m_{\psi_{l,t}(n)}}^{k_{j,l,\psi_{l,t}(n)}}\right)_{j\in J_{l,t}}$ is strict.  

Let $(i,j)\in I$. 
Note that $(x_{\psi_{l,t'}(n)})$ is a subsequence of $(x_{\psi_{l,t}(n)})$ if $t'\geq t$. 
As $\psi_{l,0}=\psi_{l-1,N}$ if $l>1$, an easy induction proves that $(x_{\psi_{l', t'}(n)})$ is a subsequence of $(x_{\psi_{l,t}(n)})$ for all $(l',t')\succeq (l,t)$.
In particular, $(x_{\psi_{b,N}(n)})$ is a subsequence of $(x_{\psi_{l,t}(n)})$.

Let $l\in \{1,\dots,b\}$ and show that the lemma holds with $J_l=J_{l,N}$ and $\psi=\psi_{b,N}$.

$a)$: By the foregoing, either $J_l$ is empty or $(v_{l,N,n})$ is strict. 
Item $a)$ follows since $(x_{\psi(n)})$ is a subsequence of $(x_{\psi_{l,N}(n)})$.

$b)$: If $j\in J_l$, then we get $b)$ by taking $\lambda^{(j,l)}=1, \lambda_j^{(j,l)} = -1$ and $\lambda_m^{(j,l)}=0$ if $m\neq j$. 
If $j\notin J_l$, then $j\notin J_{l,j}$ since $J_{l,j}\subset J_l$. 
By construction of $J_{l,j}$, it means that $u_{l,j,n}\in T_{l,j}$ for all $n$. 
Consequently, \cite[Chapter 3, §3, Theorem 5]{OnishchikVinberg} says us that there exists a tuple $\boldsymbol{\lambda}=(\lambda_m^{(j,l)})_{m\in J_{l,j}\cup \{j\}}\in \Z^{1+\#J_{l,j}}\backslash \{(0,\dots,0)\}$ such that for all $n$,  
\begin{equation} \label{eq pro}
\lambda_j^{(j,l)}k_{j,l,\psi_{l,j}(n)}+\sum_{m\in J_{l,j}} \lambda_m^{(j,l)} k_{m,l,\psi_{l,j}(n)} \equiv 0 \; (m_{\psi_{l,j}(n)}).
\end{equation}
To get $b)$, it remains to prove that $\lambda_j^{(j,l)}\neq 0$, which is clear if $J_{l,j}$ is empty since $\boldsymbol{\lambda}\neq 0$. 
If not, then the sequence $(v_{l,j,n})$ is strict. 
In particular, $v_{l,j,n}\in T_{l,j}$ for only finitely many $n$. 
Once again, \cite[Chapter 3, §3, Theorem 5]{OnishchikVinberg} tells us that the congruence $\sum_{m\in J_{l,j}} \lambda_m^{(j,l)} k_{m,l,\psi_{l,j}(n)} \equiv 0 \; (m_{\psi_{l,j}(n)})$ holds for only finitely many $n$, and \eqref{eq pro} proves that $\lambda_j^{(j,l)}\neq 0$. 
This completes the proof. 
\end{proof}

Put $O=(0,0)$.
For a point $P\in \R^2$ with affix $z$, we set $\overrightarrow{z}=\overrightarrow{OP}$. 
Next define $(\overrightarrow{z_1}, \overrightarrow{z_2})$ to be the angle formed by nonzero vectors $\overrightarrow{z_1}$ and $\overrightarrow{z_2}$. 
If $z_1=0$ or $z_2=0$, we write $(\overrightarrow{z_1}, \overrightarrow{z_2})=0$. 
We can now construct our sequence $(x_{\Phi(n)})$.
 
\begin{lmm} \label{lmm 521}
We keep the notation of Lemma \ref{cor 3}.
Put $\theta= \prod_{j=1}^N \prod_{l=1}^b \vert \lambda^{(j,l)} \vert\in\N, \Lambda_m^{(j,l)}= - \theta \lambda_m^{(j,l)}/\lambda^{(j, l)}\in\Z$ and $K_{j,l,\psi(n)}=\sum_{m\in J_l} \Lambda_m^{(j,l)} k_{m,l,\psi(n)}$.
Then there exist a subsequence $(x_{\Phi(n)})$ of $(x_{\psi(n)})$ and a subset $I\subset\{1,\dots,N\}$ such that 
\begin{enumerate} [a)]
\item for all $n$, we have \[x_{\Phi(n)}= \sum_{j\in I} a_{j,\Phi(n)}\prod_{l=1}^b \alpha_l^{K_{j,l,\Phi(n)}/(\theta m_{\Phi(n)})}= \sum_{j\in I} z_{j, \Phi(n)},\] where $a_{j,\Phi(n)}\in K_\Gamma(\mathcal{F}^{1/\theta})$ and $z_{j,\Phi(n)} = a_{j,\Phi(n)}\prod_{l=1}^b \alpha_l^{K_{j,l,\Phi(n)}/(\theta m_{\Phi(n)})}$.
\item the tuples $(K_{j,1,\Phi(n)}, \dots, K_{j,b,\Phi(n)})$ are pairwise distinct when $j$ ranges over all elements of $I$; 
\item the sequence  $(((\overrightarrow{z_{i,\Phi(n)}}, \overrightarrow{z_{j,\Phi(n)}}))_{i, j \in I})$ converges as $n\to +\infty$.
\end{enumerate} 
\end{lmm}

\begin{proof}
Let $l\in\{1,\dots,b\}$, and let $j\in\{1,\dots,N\}$. 
A small calculation involving Lemma \ref{cor 3} $b)$ gives $k_{j,l,\psi(n)}= (v_{j,l,\psi(n)} m_{\psi(n)} + K_{j,l,\psi(n)})/\theta$ for some $v_{j,l,\psi(n)}\in \Z$. 
Thus 
\begin{align*} 
x_{\psi(n)} & = \sum_{j=1}^N x_{j, \psi(n)} \prod_{l=1}^b \alpha_l^{k_{j,l,\psi(n)}/m_{\psi(n)}} \\ 
& = \sum_{j=1}^N \left(x_{j, \psi(n)} \prod_{l=1}^b \alpha_l^{v_{j,l,\psi(n)}/\theta}\right)\prod_{l=1}^b \alpha_l^{K_{j,l,\psi(n)}/(\theta m_{\psi(n)})}. 
\end{align*}
Note that $x_{j, \psi(n)} \prod_{l=1}^b \alpha_l^{v_{j,l,\psi(n)}/\theta}\in K_\Gamma(\mathcal{F}^{1/\theta})$. 
If two tuples $(K_{i,1,\psi(n)}, \dots, K_{i,b,\psi(n)})$ and $(K_{j,1,\psi(n)}, \dots, K_{j,b,\psi(n)})$ are equal, we can then group the $i$-th and the $j$-th term in the last sum above into a single. 
By repeating this process as much as possible, we construct a set $I_{\psi(n)}\subset \{1,\dots,N\}$ such that \[x_{\psi(n)}= \sum_{j\in I_{\psi(n)}} a_{j,\psi(n)}\prod_{l=1}^b \alpha_l^{K_{j,l,\psi(n)}/(\theta m_{\psi(n)})},\] where $a_{j,\psi(n)}\in K_\Gamma(\mathcal{F}^{1/\theta})$ and the tuples $(K_{j,1,\psi(n)}, \dots, K_{j,b,\psi(n)})$ are pairwise distinct when $j$ runs over all elements of $I_{\psi(n)}$.
As $I_{\psi(n)}\subset\{1,\dots, N\}$, there is a subsequence $(x_{\phi(n)})$ of $(x_{\psi(n)})$ for which the sequence $(I_{\phi(n)})$ is constant, say to $I$. 

By definition, $(\overrightarrow{z_{i,\phi(n)}}, \overrightarrow{z_{j,\phi(n)}})\in [0, 2\pi[$ for all $i,j\in I$ and all $n$. 
Hence Bolzano-Weierstrass theorem ensures us the existence of a subsequence $(x_{\Phi(n)})$ of $(x_{\phi(n)})$ such that the sequence $(((\overrightarrow{z_{i,\Phi(n)}}, \overrightarrow{z_{j,\Phi(n)}}))_{i,j \in I})$ converges as $n\to +\infty$. 
This proves $c)$. 
Finally, we directly get $a)$ and $b)$ from the construction of $I_{\Phi(n)}=I$. 
\end{proof}

\section{Proof of Theorem \ref{thm 2}} \label{section 5}
Recall that $l_N(\Gamma)$ is the set of $x\in\Q(\Gamma_{\mathrm{div}})$ that can express as $x=\sum_{j=1}^N x_j\gamma_j$ with $x_j\in K_\Gamma=\Q(\mu_\infty, \Gamma)$ and $\gamma_j\in \Gamma_{\mathrm{div}}$. 
Clearly $\tau x\in l_N(\Gamma)$ for all $\tau\in \Gal(\overline{\Q}/K_\Gamma)$ since any conjugate of $\gamma_j$ over $K_\Gamma$ is equal to $\gamma_j$ up to root of unity. 

Recall that $O_\Gamma(\alpha)$ and $d_{\Gamma, \varepsilon}(\alpha)$ have been defined in Subsection \ref{subsection 1.2}. 
Note that $d_{\Gamma, \varepsilon}(\alpha)=d_{\Gamma, \varepsilon}(\beta)$ if $\alpha$ and $\beta$ are conjugates over $K_\Gamma$, that is if $\beta\in O_\Gamma(\alpha)$. 

The goal of this section is to prove the following. 
\begin{thm} \label{thm 5}  
Let $(x_n)$ be a sequence of $l_N(\Gamma)$ such that $d_{\Gamma, \varepsilon}(y_n)\underset{n\to+\infty}{\longrightarrow} 1$ for all $\varepsilon>0$.
Let $(x_{\Phi(n)})$ be the sequence constructed in Lemma \ref{lmm 521} from which we keep the notation. 
Then
\begin{enumerate} [a)]
\item  $\sum_{j\in I} \vert z_{j,\Phi(n)}\vert^2 \to 1$;
\item there exist $\#I-1$ elements $j\in I$ such that $z_{j,\Phi(n)} \to 0$.
\end{enumerate}
\end{thm}

\textit{Proof of Theorem \ref{thm 2} by assuming Theorem \ref{thm 5}.}
Set $v_n = \max_{x\in O_\Gamma(y_n)} \{\vert  \vert x\vert^2 -1\vert\}$. 
Let $l\in\R\cup\{+\infty\}$ be an accumulation point of $(v_n)$ and show that $l=0$, which will finish the proof of our theorem. 
Without loss of generality, assume that $v_n\to l$. 

Pick $x_n\in O_\Gamma(y_n)$ such that $v_n= \vert \vert x_n\vert^2 -1 \vert$. 
By the preamble of this section, we easily infer that $x_n\in l_N(\Gamma)$ for all $n$ and $d_{\Gamma, \varepsilon}(x_n)=d_{\Gamma, \varepsilon}(y_n)\to 1$ for all $\varepsilon>0$. 
Thanks to Lemma \ref{lmm 521} $a)$, we have $x_{\Phi(n)} = \sum_{j\in I} z_{j, \Phi(n)}$. 
As a direct consequence of Theorem \ref{thm 5}, we get $\vert x_{\Phi(n)}\vert^2\to 1$.
But then $v_{\Phi(n)}\to 0$; whence $l=0$. 
\qed \\

For the rest of this section, we keep (and fix) the same notation as Theorem \ref{thm 5}.
In order to simplify our explanation, we assume that $\Phi$ is the identity.
Set $G_n=\Gal(K_\Gamma(\mathcal{F}^{1/(\theta m_n)})/K_\Gamma)$ and $H_n=\Gal(K_\Gamma(\mathcal{F}^{1/(\theta m_n)})/K_\Gamma(\mathcal{F}^{1/\theta}))$. 
 Lemma \ref{lmm 16} applied to $m=\infty, d_1=\dots=d_b=\theta m_n$ and $L=\Q(\mathcal{F}^{1/\theta})$ gives
\begin{equation} \label{eq description H n}
H_n= \prod_{l=1}^b \Z/(m_n/c_{l,n})\Z,
\end{equation}
 where $c_{1,n},\dots, c_{b,n}\in\N$ are bounded from above by a constant depending only on $\Gamma$ and $\theta$. 
Next for all $m\in \N$, set $. : \R^m \times \R^m \to \R$ to be the dot product on $\R^m$. 
Finally, put $\mathbf{K}_{j,n}=(K_{j,1,n},\dots, K_{j,b,n})$ and for any $\mathbf{r}=(r_1,\dots,r_b)\in \R^b$, we set \[B_{\mathbf{r},i,j,n} = (\overrightarrow{z_{i,n}}, \overrightarrow{z_{j,n}})+ \frac{2\pi  \; \mathbf{rc}_n.(\mathbf{K}_{j,n}-\mathbf{K}_{i,n})}{m_n}, \] where $\mathbf{rc}_n=(r_1c_{1,n}, \dots, r_bc_{b,n})$. 

\subsection{Proof of Theorem \ref{thm 5} $a)$} 
We will deduce the limit of Theorem \ref{thm 5} $a)$ thanks to the following equality. 
Recall that $x_n\in K_\Gamma(\mathcal{F}^{1/(\theta m_n)})$ by Lemma \ref{lmm 521} $a)$.

\begin{lmm} \label{lmm 12}
Let $n\in \N$, and let $\sigma=\mathbf{r}\in H_n$.
Then \[\vert \sigma x_n \vert^2 = \sum_{j\in I} \vert z_{j,n} \vert^2 + \sum_{i,j\in I, i\neq j} \vert z_{i,n} \vert \vert z_{j,n } \vert \cos\left( B_{\mathbf{r},i,j,n}\right).\]
\end{lmm}
\begin{proof}
As $\sigma\in H_n$, it therefore fixes the elements of $K_\Gamma(\mathcal{F}^{1/\theta})$. 
Let $j\in I$.
By Lemma \ref{lmm 521} $a)$, we have $z_{j,n}=a_{j,n}\prod_{l=1}^b \alpha_l^{K_{j,l,n}/m_n}$ with $a_{j,n}\in K_\Gamma(\mathcal{F}^{1/\theta})$. 
A small calculation involving \eqref{eq description H n} shows that $\sigma z_{j,n}=z_{j,n} \zeta_{m_n}^{\mathbf{rc}_n.\mathbf{K}_{j,n}}$.
From Lemma \ref{lmm 521} $a)$ and from the cosine rule, we get 
\begin{align*}
\vert \sigma x_n \vert^2 & = \left\vert \sum_{j\in I} \sigma z_{j,n}\right\vert^2 = \left\vert \sum_{j\in I} z_{j,n}\zeta_{m_n}^{\mathbf{rc_n}.\mathbf{K}_{j,n}} \right\vert^2  \\
& =  \sum_{j\in I} \vert z_{j,n} \vert^2 + \sum_{i,j\in I, i\neq j} \vert z_{i,n} \vert \vert z_{j,n}\vert \cos\left( \left(\overrightarrow{ z_{i,n}\zeta_{m_n}^{\mathbf{rc}_n.\mathbf{K}_{i,n}}}, \overrightarrow{ z_{j,n}\zeta_{m_n}^{\mathbf{rc}_n.\mathbf{K}_{j,n}}} \right)\right).
\end{align*}
To conclude, it remains to show the equality 
\begin{equation} \label{eq non repet}
\vert z_{i,n} \vert \vert z_{j,n}\vert \cos\left( \left(\overrightarrow{ z_{i,n}\zeta_{m_n}^{\mathbf{rc}_n.\mathbf{K}_{i,n}}}, \overrightarrow{ z_{j,n}\zeta_{m_n}^{\mathbf{rc}_n.\mathbf{K}_{j,n}}} \right)\right)= \vert z_{i,n} \vert \vert z_{j,n} \vert \cos\left( B_{\mathbf{r},i,j,n}\right), 
\end{equation}
which is obvious if either $ z_{i,n}=0$ or $ z_{j,n}=0$. 
If these complex numbers are nonzero, then \eqref{eq non repet} arises from the chain of equalities modulo $2\pi$ below \[(\overrightarrow{\alpha \zeta_{m_n}^{\lambda}}, \overrightarrow{\beta \zeta_{m_n}^\eta}) \equiv (\overrightarrow{\alpha}, \overrightarrow{\beta \zeta_{m_n}^{\eta-\lambda}})\equiv (\overrightarrow{\alpha}, \overrightarrow{\beta})+ (\overrightarrow{\beta}, \overrightarrow{\beta \zeta_{m_n}^{\eta-\lambda}})\equiv (\overrightarrow{\alpha}, \overrightarrow{\beta}) + \frac{2\pi(\eta-\lambda)}{m_n} \; (2\pi),\] which is valid for all $\alpha,\beta\in \C^*$ and all $\eta,\lambda\in \R$. 
\end{proof}

Fix from now $\varepsilon>0$. 
Let $F_n$ be the set of elements $\sigma\in G_n$ satisfying $1-\varepsilon \leq \vert \sigma x_n\vert ^2 \leq 1+\varepsilon$. 
It is easy to check that $ d_{\Gamma, \varepsilon}(x_n)=\# F_n/\#G_n$. 

\textit{Proof of Theorem \ref{thm 5} when $I$ has cardinality $1$.} 
Clearly $b)$ arises from $a)$. 
Let $j$ be the unique element of $I$. 
Lemma \ref{lmm 12} tells us that $\vert \sigma x_n\vert^2=\vert z_{j,n}\vert^2$ for all $\sigma\in H_n$. 
The fact that $d_{\Gamma, \varepsilon}(x_n)\to 1$ by assumption and that $H_n$ has index $[K_\Gamma(\mathcal{F}^{1/\theta}) : K_\Gamma]$ in $G_n$ provides for all $n$ large enough an element $\sigma_n\in H_n$ belonging to $F_n$. 
In particular, $1-\varepsilon \leq \vert \sigma_n x_n \vert^2 = \vert z_{j,n}\vert^2 \leq 1+\varepsilon$, and so $\vert z_{j,n}\vert^2\to 1$ proving what we desire.
\qed \\ 

We now focus on the case where $I$ has at least $2$ elements.

Lemma \ref{lmm 12} suggests us to construct for all $n$ large enough a "good" $\sigma_n=\mathbf{r}_n\in H_n$ for which we can estimate as precisely as possible the quantities $\vert \sigma_n x_n\vert $ and $\cos(B_{\mathbf{r}_n,i,j,n})$. 
Its construction is the purpose of the next proposition. 

Recall that $J_l$ is defined in Lemma \ref{cor 3}. 
Let $Y$ be the set of couples $(l,m)$ such that $l\in \{1,\dots,b\}$ and $m\in J_l$. 
It is a non-empty set. 
Indeed, otherwise $J_l$ would be empty for all $l$. 
Lemma \ref{lmm 521} then implies $K_{j,l,n}=0$ for all $j,l,n$. 
But this is possible only if $I$ has cardinality $1$ by Lemma \ref{lmm 521} $b)$, a contradiction by the foregoing.

Recall that the integer $\Lambda_m^{(i,l)}$ is defined in Lemma \ref{lmm 521}. 
Then put \[\gamma=\max_{i,j\in I, i\neq j}\left\{\sum_{(l,m)\in Y} \left \vert \Lambda_m^{(j,l)} - \Lambda_m^{(i,l)} \right\vert \right\} +1 \] 
Let $i,j\in I$.
Write $\boldsymbol{\Lambda}_{i,j}=(\Lambda_m^{(j,l)} - \Lambda_m^{(i,l)})_{(l,m)\in Y}$.
Recall that the sequence $((\overrightarrow{z_{i,n}}, \overrightarrow{z_{j,n}}))$ converges by Lemma \ref{lmm 521} $c)$; denote by $L_{i,j}$ its limit. 
Finally, for $\mathbf{x}\in \R^{\#Y}$, we write \[ I_{i,j}(\mathbf{x})=\left[e^{\zeta_4(L_{i,j}+ \boldsymbol{\Lambda}_{i,j}.\mathbf{x} -\gamma\varepsilon)} ;  e^{\zeta_4(L_{i,j}+\boldsymbol{\Lambda}_{i,j}.\mathbf{x} +\gamma\varepsilon)} \right]. \]

\begin{prop} \label{prop 3}
Let $\mathbf{x}\in \R^{\#Y}$.
Then, for all $n$ large enough, there exists $\sigma_{\mathbf{x},  n}=\mathbf{r}_{\mathbf{x},  n}\in H_n$ such that $ 1-\varepsilon \leq \vert \sigma_{\mathbf{x},  n} x_n \vert^2 \leq 1+\varepsilon$.
Furthermore, for each $i,j\in I$ distinct such that $\pm 1\notin I_{i,j}(\mathbf{x})$, we have 
 \begin{multline*}
\cos(\gamma\varepsilon)\cos(L_{i,j} + \boldsymbol{\Lambda}_{i,j}.\mathbf{x})+\sin(\gamma\varepsilon) \geq \\   \cos(B_{\mathbf{r}_{\mathbf{x},  n},i,j,n}) \geq  \cos(\gamma\varepsilon)\cos(L_{i,j} + \boldsymbol{\Lambda}_{i,j}.\mathbf{x}) - \sin(\gamma\varepsilon).
\end{multline*}
\end{prop}
\begin{proof}
\underline{First part}:
As $d_{\Gamma,\varepsilon}(x_n)\to 1$, we deduce that for all $n$ large enough,
\begin{equation} \label{eq 9}
d_{\Gamma,\varepsilon}(x_n)= \frac{\# F_n}{\# G_n} > 1-\frac{(1-\varepsilon)^b}{2[K_\Gamma(\mathcal{F}^{1/\theta}) : K_\Gamma]} \left(\frac{\varepsilon}{2\pi}\right)^{\sum_{l=1}^b \#J_l}.
\end{equation}
Put $\mathbf{x}=(x_{l,m})_{(l,m)\in Y}$. 
For $l\in\{1,\dots,b\}$, define $\mathcal{K}_{l,n}(\mathbf{x})$ to be the set $\Z/(m_n/c_{l,n})\Z$ if $J_l$ is empty and the set of $r\in \Z/(m_n/c_{l,n})\Z$ such that \[(\zeta_{m_n/c_{l,n}}^{r k_{m,l,n}})_{m\in J_l}\in \prod_{m\in J_l} [e^{\zeta_4(x_{l,m}- \varepsilon)} ; e^{\zeta_4(x_{l,m}+ \varepsilon)}]\] if $J_l$ is non-empty.
We have \[ \frac{\# \mathcal{K}_{l,n}(\mathbf{x})}{m_n/c_{l,n}}\geq (1-\varepsilon)\left(\frac{\varepsilon}{2\pi}\right)^{\# J_l} \] for all $n$ large enough. 
It is clear if $J_l$ is empty. 
If not, then by Lemma \ref{cor 3} $a)$, the sequence of terms $(\zeta_{m_n}^{k_{m,l,n}})_{m\in J_l}$ is strict. 
As $c_{l,n}$ is bounded from above by a constant depending only on $\Gamma$ and $\theta$, we deduce that the sequence of terms $(\zeta_{m_n/c_{l,n}}^{k_{m,l,n}})_{m\in J_l}$ is also strict. 
The desired inequality now arises from Corollary \ref{lmm 9} applied to this sequence, $M=\#J_l$ and $(x_1,\dots,x_M)=(x_{l,m})_{m\in J_l}$. 
 
Clearly, $\prod_{l=1}^b \mathcal{K}_{l,n}(\mathbf{x})$ is a subset of $H_n$ by \eqref{eq description H n}.
Moreover, $H_n$ has cardinality $\prod_{l=1}^b m_n/c_{l,n}$ and has index $[K_\Gamma(\mathcal{F}^{1/\theta}) : K_\Gamma]$ in $G_n$.  
Thus
\begin{align*} 
\frac{\# \prod_{l=1}^b \mathcal{K}_{l,n}(\mathbf{x})}{\#G_n} & = \frac{1}{[K_\Gamma(\mathcal{F}^{1/\theta}) : K_\Gamma]}\prod_{l=1}^b \frac{\# \mathcal{K}_{l,n}(\mathbf{x})}{m_n/c_{l,n}}\geq \frac{(1-\varepsilon)^b }{[K_\Gamma(\mathcal{F}^{1/\theta}) : K_\Gamma]} \left(\frac{\varepsilon}{2\pi}\right)^{\sum_{l=1}^b \#J_l}
\end{align*}
 for all $n$ large enough. 
Combining this inequality and \eqref{eq 9} provides an element $\sigma_n =\mathbf{r}=(r_1,\dots,r_b)\in F_n\cap \prod_{l=1}^b \mathcal{K}_{l,n}(\mathbf{x})$. 
This ends the first part. 

\underline{Second part}:  
Let $i,j\in I$ be as in the statement.
A small calculation gives 
\begin{align*}
\zeta_{m_n}^{\mathbf{rc}_n.(\mathbf{K}_{j,n}-\mathbf{K}_{i,n})} & = \zeta_{m_n}^{\sum_{l=1}^b r_lc_{l,n} (K_{j,l,n}-K_{i,l,n})} = \zeta_{m_n}^{\sum_{l=1}^b \sum_{m\in J_l}r_lc_{l,n}(\Lambda_m^{(j,l)} - \Lambda_m^{(i,l)})k_{m,l,n}} \\
& = \prod_{(l,m)\in Y} (\zeta_{m_n}^{ r_lc_{l,n}k_{m,l,n}})^{\Delta_m^{(i,j,l)}}, 
\end{align*}
where $\Delta_m^{(i,j,l)}=\Lambda_m^{(j,l)} - \Lambda_m^{(i,l)}\in \Z$. 
Let $(l,m)\in Y$. 
In particular, $J_l$ is non-empty. 
Recall that $r_l\in \mathcal{K}_{l,n}(\mathbf{x})$. 
By definition of $\mathcal{K}_{l,n}(\mathbf{x})$, we have \[ \zeta_{m_n}^{r_lc_{l,n}k_{m,l,n}}= \zeta_{m_n/c_{l,n}}^{r_lk_{m,l,n}}\in [e^{\zeta_4 (x_{l,m}-\varepsilon)}; e^{\zeta_4(x_{l,m} + \varepsilon)}]\] for all $m\in J_l$. 
Thus
\begin{multline*} 
\zeta_{m_n}^{\mathbf{rc}_n.(\mathbf{K}_{j,n}-\mathbf{K}_{i,n})} \in \\
\left[\prod_{(l,m)\in Y} e^{\zeta_4 \left(\Delta_m^{(i,j,l)}x_{l,m} -  \varepsilon \left\vert \Delta_m^{(i,j,l)}  \right\vert\right)} ;  \prod_{(l,m)\in Y} e^{\zeta_4\left(\Delta_m^{(i,j,l)} x_{l,m} + \varepsilon \left\vert \Delta_m^{(i,j,l)}\right\vert \right)} \right]. 
\end{multline*} 
We clearly have  \[ \sum_{(l,m)\in Y} \Delta_m^{(i,j,l)}x_{l,m} = \sum_{(l,m)\in Y} (\Lambda_m^{(j,l)} - \Lambda_m^{(i,l)})x_{l,m} = \boldsymbol{\Lambda}_{i,j}.\mathbf{x} \]
and we finally conclude from the definition of $\gamma$ that \[\zeta_{m_n}^{\mathbf{rc}_n.(\mathbf{K}_{j,n}-\mathbf{K}_{i,n})}\in [e^{\zeta_4 (\boldsymbol{\Lambda}_{i,j}.\mathbf{x}-(\gamma-1)\varepsilon)} ; e^{\zeta_4(\boldsymbol{\Lambda}_{i,j}.\mathbf{x}+(\gamma-1)\varepsilon)}].\]
On the other hand, $e^{\zeta_4(\overrightarrow{z_{i,n}}, \overrightarrow{z_{j,n}})}\in [e^{\zeta_4(L_{i,j}-\varepsilon)},e^{\zeta_4( L_{i,j}+\varepsilon)}]$ for all $n$ large enough since $(\overrightarrow{z_{i,n}}, \overrightarrow{z_{j,n}}) \to L_{i,j}$. 
From all this, we conclude that for all $n$ large enough, \[ e^{\zeta_4 B_{\mathbf{r}, i,j,n}} = e^{\zeta_4 (\overrightarrow{z_{i,n}}, \overrightarrow{z_{j,n}})} \zeta_{m_n}^{\mathbf{rc}_n.(\mathbf{K}_{j,n}-\mathbf{K}_{i,n})}\in I_{i,j}(\mathbf{x}).\]
By assumption, $\pm 1 \notin I_{i,j}(\mathbf{x})$.
The real part function is therefore monotone on $I_{i,j}(\mathbf{x})$. 
In particular, it reaches its extrema to the two endpoints of $I_{i,j}(\mathbf{x})$.
If it is decreasing, then \[ \cos(L_{i,j}+ \boldsymbol{\Lambda}_{i,j}.\mathbf{x}-\gamma\varepsilon) \geq \cos(B_{\mathbf{r},i,j,n}) \geq \cos(L_{i,j}+ \boldsymbol{\Lambda}_{i,j}.\mathbf{x}+\gamma\varepsilon). \]
We deduce the desired inequality thanks to the relations $\cos(x-y)=\cos(x)\cos(y)+\sin(x)\sin(y)$, $\cos(x+y) = \cos(x)\cos(y)- \sin(x)\sin(y)$ and $\vert \sin(x)\vert \leq 1$. 
The increasing case is similar, which proves the second part of the statement.
\end{proof}

The key to get Theorem \ref{thm 5} $a)$ is to apply the previous proposition to a finite number of well-chosen $\mathbf{x}$.
But before constructing them, we need some preliminary results. 

\begin{lmm} \label{lmm 22distinct}
Let $i,j\in I$ be distinct. 
Then $\boldsymbol{\Lambda}_{i,j}$ is non-zero. 
\end{lmm}

\begin{proof}
If $\boldsymbol{\Lambda}_{i,j}$ is zero, then $\Lambda_m^{(j,l)}=\Lambda_m^{(i,l)}$ for all $(l,m)\in Y$, i.e. for all $l\in\{1,\dots,b\}$ and all $m\in J_l$. 
But then, Lemma \ref{lmm 521} implies $K_{i,l,n}=K_{j,l,n}$ for all $l\in\{1,\dots,b\}$ and all $n$, i.e. $(K_{i,1,n},\dots,K_{i,b,n})=(K_{j,1,n},\dots, K_{j,b,n})$ for all $n$. 
This is possible only if $i=j$ according to Lemma \ref{lmm 521} $b)$, a contradiction.
The lemma follows. 
\end{proof}

We can thus fix an integer $d\in \N$ such that for all $i,j\in I$ distinct, $d$ does not divide at least one of the coordinates of $\boldsymbol{\Lambda}_{i,j}$.  
Put \[Z= \{ 0; 2\pi /d;\dots; 2\pi (d-1)/d \}^{\#Y}.\]
\begin{lmm} \label{lmm 10}
Let $i,j\in I$ be distinct.
Then $\sum_{\mathbf{z}\in Z} e^{\zeta_4 \boldsymbol{\Lambda}_{i,j}.\mathbf{z}}=0$.  
\end{lmm} 
\begin{proof}
For brevity, write $\boldsymbol{\Lambda}_{i,j}=(\lambda_1,\dots, \lambda_{\#Y})$. 
Thus \[\sum_{\mathbf{z}\in Z} e^{\zeta_4 \boldsymbol{\Lambda}_{i,j}.\mathbf{z}} = \sum_{0\leq k_1,\dots, k_{\#Y}\leq d-1} e^{\zeta_4 \sum_{t=1}^{\#Y} \lambda_t 2\pi k_t/d} = \sum_{0\leq k_1,\dots, k_{\#Y}\leq d-1} \zeta_d^{\sum_{t=1}^{\#Y} \lambda_t k_t}. \] 
By definition of $d$, there is $u\in\{1,\dots,\# Y\}$ such that $d$ does not divide $\lambda_u$. 
So \[\sum_{\mathbf{z}\in Z} e^{\zeta_4 \boldsymbol{\Lambda}_{i,j}.\mathbf{z}} = \sum_{0\leq k_1,\dots, k_{u-1}, k_{u+1}, \dots, k_{\#Y}\leq d-1} \zeta_d^{\sum_{\underset{t\neq u}{t=1}}^{\#Y} \lambda_t k_t} \sum_{k_u=0}^{d-1} \zeta_d^{\lambda_u k_u}. \] 
Finally, $\sum_{k=0}^{d-1} \zeta_d^{\lambda_u k}=0$ since $d$ does not divide $\lambda_u$. 
The lemma follows.
\end{proof}
We would like to apply Proposition \ref{prop 3} for all $\mathbf{x}\in Z$.
However, there is no guarantee that the condition $\pm 1\notin I_{i,j}(\mathbf{x})$ holds for all $i,j\in I$ distinct and all $\mathbf{x}\in Z$.
We see below how to circumvent this difficulty. 

We define $\mathcal{H}$ to be the set of $\mathbf{X} \in [0;2\pi]^{\# Y}$ satisfying \[ \exists t\in\Z, \exists i\neq j\in I,  \exists \mathbf{z} \in Z \; \text{such that} \; L_{i,j}+ \boldsymbol{\Lambda}_{i,j}. (\mathbf{X}+\mathbf{z}) - t\pi =0. \]
As $\boldsymbol{\Lambda}_{i,j}$ is non-zero for all $i,j\in I$ distinct, we infer that $\mathcal{H}$ lies in a finite union of hyperplanes in $\R^{\#Y}$ (the equation above having no solution if $\vert t \vert$ is large enough since the quantity $\mathbf{X}$ is bounded).
Hence there exists a simply-connected compact $K \subset [0;2\pi]^{\#Y}$ such that $K \cap \mathcal{H}$ is empty.
The distance $\delta$ from $K$ to $\mathcal{H}$ is a positive real since both $K$ and $\mathcal{H}$ are compact.
For $\mathbf{x}_0\in K$, the distance from $\mathbf{x}_0$ to $\mathcal{H}$ is \[\min_{t\in \Z, i\neq j\in I, \mathbf{z}\in Z}\left\{ \frac{\vert L_{i,j}+ \boldsymbol{\Lambda}_{i,j}.(\mathbf{x}_0+\mathbf{z}) - t\pi \vert}{\sqrt{\boldsymbol{\Lambda}_{i,j}.\boldsymbol{\Lambda}_{i,j}}}\right\}\geq \delta. \] 
As $\boldsymbol{\Lambda}_{i,j}\in \Z^{\#Y}$ is non-zero, we conclude $\sqrt{\boldsymbol{\Lambda}_{i,j}.\boldsymbol{\Lambda}_{i,j}}\geq 1$, and so \[\vert L_{i,j}+ \boldsymbol{\Lambda}_{i,j}.\mathbf{x} - t\pi \vert \geq \delta\] for all $t\in \Z$, all $i,j\in I$ distinct and all $\mathbf{x}\in K+Z= \{ \mathbf{X}+\mathbf{z}, \mathbf{X}\in K, \mathbf{z}\in Z\}$.

Recall that $\varepsilon$ is as small as possible. 
So we can take it such that $\gamma \varepsilon < \delta$.
Thus \[ t\pi \notin \left[L_{i,j}+ \boldsymbol{\Lambda}_{i,j}.\mathbf{x} -\gamma\varepsilon ;  L_{i,j}+\boldsymbol{\Lambda}_{i,j}.\mathbf{x} +\gamma\varepsilon \right]\] for all $t\in\Z$, all $i,j\in I$ distinct and all $\mathbf{x}\in K+Z$.  
In conclusion, $\pm 1\notin I_{i,j}(\mathbf{x})$ for all $i,j\in I$ distinct and all $\mathbf{x}\in K+Z$.  
 
Here is the last calculation before starting the proof of Theorem \ref{thm 5} $a)$. 

\begin{lmm} \label{lmm 11}
Let $\eta, x_j\in \C$ be complex numbers with $j\in I$. 
Then \[ \sum_{j\in I} x_j^2 + \eta \sum_{i,j\in I, i\neq j} x_ix_j = -\frac{\eta}{2}\sum_{i,j\in I, i\neq j} (x_i-x_j)^2 + (1+\eta(\#I-1))\sum_{j\in I} x_j^2.   \]
\end{lmm}
\begin{proof}
For brevity, put $\eta_1=-\eta/2$ and $\eta_2=1+\eta(\#I-1)$. 
Then 
\begin{align*}
\eta_1\sum_{i,j\in I, i\neq j} (x_i-x_j)^2 + \eta_2\sum_{j\in I} x_j^2 & = \eta_1\sum_{i,j\in I, i\neq j} (x_i^2+x_j^2-2x_ix_j) +\eta_2\sum_{j\in I} x_j^2 \\ 
& = 2(\#I-1)\eta_1\sum_{i\in I} x_i^2 -2\eta_1\sum_{i,j\in I, i\neq j} x_ix_j +\eta_2 \sum_{j\in I} x_j^2 \\ 
& = (\eta_2+2(\#I-1)\eta_1) \sum_{j\in I} x_j^2 -2\eta_1\sum_{i,j\in I, i\neq j} x_ix_j
\end{align*}
and the lemma follows since $-2\eta_1=\eta$ and $\eta_2+2(\#I-1)\eta_1=1$. 
\end{proof}

\textit{Proof of Theorem \ref{thm 5} $a)$.}  
Let $\mathbf{y}\in K$. 
Recall that $\pm 1 \notin I_{i,j}(\mathbf{x})$ for all $i,j\in I$ distinct and all $\mathbf{x}\in \mathbf{y}+Z$. 
The set $\mathbf{y}+Z$ being finite, we infer that for all $n$ large enough, Proposition \ref{prop 3} holds for all elements $\mathbf{x}\in \mathbf{y}+Z$. 
Choose such a $n$. 

Lemma \ref{lmm 10} easily implies $\sum_{\mathbf{x}\in \mathbf{y}+Z} \cos(L_{i,j}+ \mathbf{\Lambda}_{i.j}.\mathbf{x})=0$ for all $i,j\in I$ distinct. 
Summing over $\mathbf{y}+Z$ the chain of inequalities in Proposition \ref{prop 3} leads to \[ \sin(\gamma\varepsilon) \#Z \geq \sum_{\mathbf{x}\in \mathbf{y}+Z} \cos(B_{\mathbf{r}_{\mathbf{x},  n},i,j,n}) \geq -\sin(\gamma\varepsilon)\#Z. \]
Proposition \ref{prop 3} gives $1+\varepsilon \geq \vert \sigma_{\mathbf{x},  n} x_n\vert^2 \geq 1-\varepsilon$ for all $\mathbf{x}\in \mathbf{y}+ Z$. 
By Lemma \ref{lmm 12},  \[ 1+\varepsilon \geq \sum_{j\in I} \vert z_{j,n}\vert^2 + \sum_{i,j\in I, i\neq j} \vert z_{i,n}\vert \vert z_{j,n}\vert \cos(B_{\mathbf{r}_{\mathbf{x},  n},i,j,n}) \geq 1-\varepsilon. \]
By summing these inequalities over $\mathbf{y}+Z$, we conclude \[ (1+\varepsilon)\#Z \geq \#Z\sum_{j\in I} \vert z_{j,n}\vert^2 -\sin(\gamma\varepsilon)\#Z \sum_{i,j\in I, i\neq j} \vert z_{i,n}\vert \vert z_{j,n}\vert \] and \[ \#Z\sum_{j\in I} \vert z_{j,n}\vert^2 +\sin(\gamma\varepsilon)\#Z \sum_{i,j\in I, i\neq j} \vert z_{i,n}\vert \vert z_{j,n}\vert \geq (1-\varepsilon)\#Z. \]
Finally, Lemma \ref{lmm 11} applied to $x_j=\vert z_{j,n}\vert$ and $\eta=\pm \sin(\gamma\varepsilon)$ gives us
\begin{align*}
1+\varepsilon & \geq \frac{\sin(\gamma\varepsilon)}{2}\sum_{i,j\in I, i\neq j}(\vert z_{i,n}\vert - \vert z_{j,n} \vert)^2 + \left(1-(\#I-1)\sin(\gamma\varepsilon)\right)\sum_{j\in I} \vert z_{j,n}\vert^2 \\
&\geq  \left(1-(\#I-1)\sin(\gamma\varepsilon)\right)\sum_{j\in I} \vert z_{j,n}\vert^2
\end{align*}
and
\begin{align*}
1-\varepsilon & \leq -\frac{\sin(\gamma\varepsilon)}{2}\sum_{i,j\in I, i\neq j}(\vert z_{i,n}\vert - \vert z_{j,n} \vert)^2 + \left(1+(\#I-1)\sin(\gamma\varepsilon)\right)\sum_{j\in I} \vert z_{j,n}\vert^2 \\
& \leq \left(1+(\#I-1)\sin(\gamma\varepsilon)\right)\sum_{j\in I} \vert z_{j,n}\vert^2
\end{align*}
In conclusion, for all $n$ large enough, we have \[ \frac{1-\varepsilon}{1+(\#I-1)\sin(\gamma\varepsilon)} \leq \sum_{j\in I} \vert z_{j,n}\vert^2 \leq \frac{1+\varepsilon}{1-(\#I-1)\sin(\gamma\varepsilon)},  \] i.e. $\sum_{j\in I} \vert z_{j,n}\vert^2 \to 1$, which proves Theorem \ref{thm 5} $a)$.
\qed
 
\subsection{Proof of Theorem \ref{thm 5} $b)$}
Recall that $\pm 1 \notin I_{i,j}(\mathbf{x})$ for all $i,j\in I$ distinct and all $\mathbf{x}\in K\subset K+Z$. 
 We can now show the pointwise limit below.
\begin{lmm} \label{lmm 13}
Pick $\mathbf{x}\in K$. 
Then $\sum_{i,j\in I, i\neq j} \vert z_{i,n} \vert \vert z_{j,n} \vert \cos\left( L_{i,j} + \boldsymbol{\Lambda}_{i,j}.\mathbf{x} \right) \to 0$.
\end{lmm}
\begin{proof}
Thanks to Theorem \ref{thm 5} $a)$, we have $1-\varepsilon \leq \sum_{j\in I} \vert z_{j,n} \vert^2 \leq 1+\varepsilon$ for all $n$ large enough. 
Furthermore, for all $n$ large enough, there is $\sigma_{\mathbf{x},n}\in H_n$ as in Proposition \ref{prop 3}. 
Choose $n$ large enough so that the facts above hold. 

Using the triangle inequality, then Proposition \ref{prop 3}, we get 
\begin{multline} \label{eq 10}
\left\vert \sum_{i,j\in I, i\neq j} \vert z_{i,n} \vert \vert z_{j,n } \vert \left(\cos (B_{\mathbf{r_{\mathbf{x},n}},i,j,n})  -\cos(\gamma\varepsilon) \cos\left(L_{i,j}+\boldsymbol{\Lambda}_{i,j}.\mathbf{x}\right)\right) \right\vert \leq \\
\sum_{i,j\in I, i\neq j} \vert z_{i,n} \vert \vert z_{j,n}\vert \sin(\gamma\epsilon)\leq \gamma\varepsilon \sum_{i,j\in I, i\neq j} \vert z_{i,n} \vert \vert z_{j,n } \vert. 
\end{multline}
We also have $1-\varepsilon \leq \vert \sigma_{\mathbf{x},n} x_n\vert^2 \leq 1+\varepsilon$ by Proposition \ref{prop 3}.
Recall that we have the chain of inequalities $1-\varepsilon \leq \sum_{j\in I} \vert z_{j,n} \vert^2 \leq 1+\varepsilon$.
Thanks to Lemma \ref{lmm 12}, we obtain
\begin{equation} \label{eq n imp pour pas 2 fois le meme}
 \left\vert \sum_{i,j\in I, i\neq j} \vert z_{i,n} \vert \vert z_{j,n } \vert \cos (B_{\mathbf{r}_{\mathbf{x},n},i,j,n}) \right\vert \leq 2\varepsilon. 
\end{equation}
Applying the reverse triangle inequality to \eqref{eq 10}, it follows from \eqref{eq n imp pour pas 2 fois le meme} that \[\cos(\gamma\varepsilon) \left\vert \sum_{i,j\in I, i\neq j} \vert z_{i,n} \vert \vert z_{j,n } \vert \cos\left(L_{i,j}+\boldsymbol{\Lambda}_{i,j}.\mathbf{x}\right)  \right\vert \leq \varepsilon\left(\gamma \sum_{i,j\in I, i\neq j}  \vert z_{i,n} \vert \vert z_{j,n } \vert  +2 \right). \]
As $\vert z_{j,n} \vert^2\leq 1+\varepsilon$ for all $j\in I$, we finally conclude that for all $n$ large enough, \[ \left\vert \sum_{i,j\in I, i\neq j} \vert z_{i,n} \vert \vert z_{j,n } \vert \cos\left(L_{i,j}+\boldsymbol{\Lambda}_{i,j}.\mathbf{x}\right)  \right\vert \leq \frac{\varepsilon}{\cos(\gamma\varepsilon)} (2+\gamma(1+\varepsilon) (\# I)^2) ,\] which ends the proof of the lemma.  
\end{proof}
Let $\{\mathbf{c}_1,\dots, \mathbf{c}_M\}=\{\boldsymbol{\Lambda}_{i,j}, i, j\in I, i\neq j\}$ be with $\mathbf{c}_1,\dots,\mathbf{c}_M$ pairwise distinct. 
For $k\in \{1,\dots, M\}$, put $E_k$ the set of $i,j\in I$ distinct such that $\mathbf{c}_k=\boldsymbol{\Lambda}_{i,j}$. 
Clearly $E_1,\dots, E_k$ is a partition of $I^2\backslash \bigsqcup_{l\in I} (l,l)$. 

We now state a much more precise result than Lemma \ref{lmm 13}.  
\begin{lmm}  \label{lmm 14}
We have $\sum_{(i,j)\in E_k} \vert z_{i,n}\vert \vert z_{j,n}\vert  e^{\zeta_4 L_{i,j} } \to 0$ for all $k\in\{1,\dots,M\}$.
\end{lmm}
\begin{proof}
Write \[C_{k,n}=\sum_{(i,j)\in E_k} \vert z_{i,n}\vert \vert z_{j,n}\vert\cos(L_{i,j}) \; \; \; \; \; \text{and} \; \; \; \; \; S_{k,n}=-\sum_{(i,j)\in E_k} \vert z_{i,n}\vert \vert z_{j,n}\vert\sin(L_{i,j}).\] 
 The sequence of terms $(C_{k,n}, S_{k,n})_{k=1}^M$ has an accumulation point in $\R^{2M}$ since it is bounded by Theorem \ref{thm 5} $a)$.
Let $(C_k,S_k)_{k=1}^M$ be such a point.
To show our lemma, it is sufficient to get $C_k=S_k=0$ for all $k\in\{1,\dots,M\}$.
Without loss of generality, assume that $(C_{k,n}, S_{k,n})_{k=1}^M\to (C_k, S_k)_{k=1}^M$.

Let $\mathbf{x}\in K$. 
A short calculation gives
\begin{align*}
 \sum_{i,j\in I, i\neq j} \vert z_{i,n}\vert \vert z_{j,n}\vert \cos\left(L_{i,j}+ \boldsymbol{\Lambda}_{i,j}.\mathbf{x}\right) & =
\sum_{k=1}^M \sum_{(i,j)\in E_k}  \vert z_{i,n}\vert \vert z_{j,n}\vert \cos\left(L_{i,j}+ \mathbf{c}_k.\mathbf{x} \right) \\ 
& = \sum_{k=1}^M C_{k,n}\cos\left( \mathbf{c}_k.\mathbf{x}\right)  + S_{k,n}\sin\left(\mathbf{c}_k.\mathbf{x} \right).
\end{align*}
We deduce from Lemma \ref{lmm 13} that $\sum_{k=1}^M C_{k,n}\cos\left( \mathbf{c}_k.\mathbf{x}\right)  + S_{k,n}\sin\left(\mathbf{c}_k.\mathbf{x} \right) \to 0$
and the uniqueness of the limit gives \[ \sum_{k=1}^M C_k \cos\left( \mathbf{c}_k.\mathbf{x} \right) + S_k \sin\left(\mathbf{c}_k.\mathbf{x} \right) =0.\]
As $K$ is a simply-connected compact, the Monodromy Theorem claims that this equality holds for all $\mathbf{x}\in \C^{\#Y}$. 
Thus, for such a $\mathbf{x}$, we have \[ \sum_{k=1}^M C_k\cos\left( \mathbf{c}_k.\mathbf{x}\right)= \sum_{k=1}^M S_k\sin\left( \mathbf{c}_k.\mathbf{x}\right)=0\] since the functions $\mathbf{x}\mapsto \sum_{k=1}^M C_k\cos\left( \mathbf{c}_k.\mathbf{x}\right)$ and $\mathbf{x}\mapsto \sum_{k=1}^M S_k\sin\left( \mathbf{c}_k.\mathbf{x}\right)$ are both even and odd. 
The tuples $\mathbf{c}_1, \dots, \mathbf{c}_M$ being pairwise distinct by construction, we get $C_k=S_k=0$ for all $k\in\{1,\dots,M\}$. 
The lemma follows.
\end{proof}
The order to compare the elements in $\R^{\# Y}$ is the lexicographical order.
For all $j\in I$, put $\boldsymbol{\Lambda}_j=(\Lambda_m^{(j,l)})_{(l,m)\in Y}$ and note that $\boldsymbol{\Lambda}_{i,j}=\boldsymbol{\Lambda}_j-\boldsymbol{\Lambda}_i$. 
Thanks to Lemma \ref{lmm 22distinct}, it follows that $\boldsymbol{\Lambda}_j$ are pairwise distinct when $j$ runs over all elements of $I$. 

\textit{Proof of Theorem \ref{thm 5} $b)$}: 
Write $E$ for the set of elements $j\in I$ such that the sequence $(z_{j,n})$ does not go to $0$. 
Theorem \ref{thm 5} $a)$ claims that $E$ is non-empty.
To get $b)$, it suffices to prove that $\#E=1$. 
Assume by contradiction that $\# E >1$. 
Let $i_0, j_0\in E$ be distinct such that $\boldsymbol{\Lambda}_{j_0}=\max_{h\in E} \boldsymbol{\Lambda}_h$ and $\boldsymbol{\Lambda}_{i_0}=\min_{h\in E} \boldsymbol{\Lambda}_h$.

Let $k\in\{1,\dots, M\}$ be the unique integer such that $(i_0,j_0)\in E_k$.
Lemma \ref{lmm 14} gives $\sum_{(i,j)\in E_k} \vert z_{i,n}\vert \vert z_{j,n} \vert e^{\zeta_4 L_{i,j}} \to 0$.
As the sequence of terms $(z_{j,n})_{j\in I}$ is bounded by Theorem \ref{thm 5} $a)$, we get $\vert z_{i,n}\vert \vert z_{j,n}\vert \to 0$ if either $i\notin E$ or $j\notin E$.
From the equality \[E_k = (E_k \cap E^2) \sqcup \{(i,j)\in E_k, i\notin E \; \text{or} \; j\notin E\},\] we infer that $\sum_{(i,j)\in E_k \cap E^2} \vert z_{i,n}\vert \vert z_{j,n} \vert e^{\zeta_4 L_{i,j}} \to 0$.

Let $(i,j)\in E_k \cap E^2$. 
As $(i_0, j_0)\in E_k \cap E^2$, we have $\boldsymbol{\Lambda}_{i_0,j_0}= \mathbf{c}_k = \boldsymbol{\Lambda}_{i,j}$; whence \[\boldsymbol{\Lambda}_{j}- \boldsymbol{\Lambda}_{i}=\boldsymbol{\Lambda}_{i,j}= \boldsymbol{\Lambda}_{i_0,j_0} = \boldsymbol{\Lambda}_{j_0} - \boldsymbol{\Lambda}_{i_0}.\]
The maximality of $\boldsymbol{\Lambda}_{j_0}$, together with the minimality of $\boldsymbol{\Lambda}_{i_0}$, shows that $\boldsymbol{\Lambda}_j= \boldsymbol{\Lambda}_{j_0}$ and $\boldsymbol{\Lambda}_i= \boldsymbol{\Lambda}_{i_0}$. 
Since $\boldsymbol{\Lambda}_j$ are pairwise distinct when $j$ ranges over all elements of $I$, we deduce that $(i,j)=(i_0,j_0)$, and so $E_k \cap E^2=\{(i_0,j_0)\}$.
But then, $\vert z_{i_0,n} \vert \vert z_{j_0,n}\vert e^{\zeta_4 L_{i_0,j_0}} \to 0$, i.e. either $z_{i_0,n} \to 0$ or $z_{j_0,n} \to 0$, contradicting the definition of $E$. 
 Theorem \ref{thm 5} $b)$ follows.
\qed  

\bibliographystyle{plain}

\end{document}